\documentclass[review]{elsarticle}

\journal{Journal} 

\usepackage{amsmath,
	amssymb
}

\makeatletter
\@addtoreset{equation}{section}

\makeatother

\newtheorem{theorem}{Theorem}[section]

\newtheorem{lemma}[theorem]{Lemma}

\newtheorem{proposition}[theorem]{Proposition}

\newenvironment{proof}{\par \vspace{0.3cm} \noindent{\sc Proof:} \ignorespaces}%
{\nolinebreak\hfill $\square$\par \medskip}
\newcommand{\dist}{\mathop{\mathrm{dist}}\limits}

\newcommand{\bnull}{{\bf 0}}
\newcommand{\bff}{{\bf f}}

\newcommand{\bw}{{\bf w}}

\newcommand{\bq}{{\bf q}}
\newcommand{\bc}{{\bf c}}

\newcommand{\be}{{\bf e}}
\newcommand{\bu}{{\bf u}}

\newcommand{\ba}{{\bf a}}

\newcommand{\bp}{{\bf p}}

\newcommand{\br}{{\bf r}}

\newcommand{\bz}{{z}}
\newcommand{\bx}{{\bf x}}

\newcommand{\tpsi}{{\boldsymbol \phi}}

\newcommand{\tA}{\widetilde{A}}

\newcommand{\by}{{\bf y}}
\newcommand{\bv}{{\bf v}}

\newcommand{\balpha}{{\boldsymbol \alpha}}
\newcommand{\bbeta}{{\boldsymbol \beta}}

\newcommand{\bphi}{{\boldsymbol \phi}}
\newcommand{\bpsi}{{\boldsymbol \psi}}

\newcommand{\spann}{\operatorname{span}}

\def\S{ \mathbb {S}}

\begin{document}
	
	\begin{frontmatter}
		
		\title{On a matrix-valued PDE characterizing \\ a contraction metric 
		for a periodic orbit}
%
		\author{Peter Giesl}
			\ead{p.a.giesl@sussex.ac.uk}
		\address{Department of Mathematics,	University of Sussex,
			Falmer BN1 9QH, UK}

		\begin{abstract}The stability and the basin of attraction of a periodic orbit can be determined using a contraction metric, i.e., a Riemannian metric with respect to which adjacent solutions contract. A contraction metric does not require knowledge of the position of the periodic orbit and is robust to perturbations.
			
			In this paper we characterize such a Riemannian contraction metric as matrix-valued solution of a linear first-order Partial Differential Equation. This will enable the explicit construction of a contraction metric by numerically solving this equation in future work. In this paper we prove existence and uniqueness of the solution of the PDE and show that it defines a contraction metric. 
		\end{abstract}
		
		
	\end{frontmatter}
	

\section{Introduction}

Ordinary differential equations arise in many important applications and the determination of periodic orbits, their stability and basins of attraction are important tasks. We consider a general autonomous Ordinary Differential Equation (ODE) of the form
$$\dot{\bx}=\bff(\bx),$$
where $\bff\colon \mathbb R^n\to \mathbb R^n$ is sufficiently smooth.

The stability and the basin of attraction of a periodic orbit can be determined using a Lyapunov function, however, its definition requires the exact position of the periodic orbit. A contraction metric can show the existence, uniqueness and stability of a periodic orbit without knowledge of its position. Moreover, a contraction metric is robust to small perturbations of the system or the metric, which ensures that even a good approximation to a contraction metric, e.g. using numerical methods, is itself a  contraction metric.

A contraction metric is a Riemannian metric such that the distance between adjacent trajectories decreases over time with respect to the Riemannian metric. Such solutions are also called incrementally stable and a contraction metric is a special type of a Finsler-Lyapunov function \cite{forni}. The contraction condition can be formulated as a local condition in a point $\bx\in\mathbb R^n$ and all adjacent solutions through $\bx+\bv$ for small $\bv\in\mathbb R^n$. If the contraction condition holds for all points $\bx$ in a compact, positively invariant and connected set $K$, then there exists one and only one attractor in $K$, it is exponentially stable and $K$ is a subset of its basin of attraction.
If the contraction holds for all adjacent directions $\bv$, then the attractor is an equilibrium. If the contraction only holds for $\bv$ perpendicular to  $\bff(\bx)$ and $K$ does not contain any equilibrium, then the attractor is a periodic orbit, see Theorem \ref{1.1}.

 Contraction metrics for periodic orbits have been studied by Borg \cite{borg} with the Euclidean metric and Stenstr\"om \cite{stenstroem} with a general Riemannian metric.
 Further results using a contraction metric have been obtained in   \cite{hartman,hartmanbook,leonov3,leonov96}.

Converse theorems, showing the existence of a contraction metric defined in the basin of attraction of an exponentially stable periodic orbit, have been obtained in \cite{giesl04}.
\cite[Section 3.5]{lohmiller} gave a converse theorem, but the Riemannian metric $M(t,\bx)$ depends on $t$ and, in general, can become  unbounded as $t\to\infty$.
In \cite[Theorem 3]{manchester}, the authors have expressed a transverse contraction condition, i.e. a contraction metric for periodic orbits, using Linear Matrix Inequalities and have used SOS (sum of squares) to construct it.


In the case of contraction metrics for an equilibrium, converse theorems have been established in \cite{converse}, characterizing the contraction metric as solution of a matrix-valued PDE. Hence, an approximate solution to the PDE, e.g. using numerical methods \cite{giesl-wendland}, constructs a contraction metric.

In this paper we seek to establish a similar result for contraction metrics for periodic orbits. 
The non-trivial challenge is to restrict the space of adjacent solutions in direction $\bv$ to an $(n-1)$-dimensional hyperplane by using a projection onto it. 
%

Let us give an outline of the contents: in Section \ref{suf} we define a contraction metric, show that it provides a sufficient condition for the existence, uniqueness and exponential stability of a periodic orbit and determines its basin of attraction. Furthermore, we show that the solution of a matrix-valued PDE defines such a contraction metric.
In Section \ref{ex} we prove the existence of a solution of the above matrix-valued PDE and in Section \ref{uni} we prove its uniqueness.  We conclude in Section \ref{con}.

\section{Sufficiency}
\label{suf}

Let us consider the ODE
\begin{eqnarray}
	\dot{\bx}&=&\bff(\bx)\label{ODE}
	\end{eqnarray}
	where $\bff\in C^\sigma(\mathbb R^n,\mathbb R^n)$, $n\in \mathbb N$ and $\sigma\ge 1$.
	A Riemannian metric is a matrix-valued function $M\in C^1(D,\mathbb S^n)$, where $D\subset \mathbb R^n $ is a domain and $\mathbb S^n$ denotes the symmetric $\mathbb R^{n\times n}$ matrices, such that $M(\bx)$ is positive definite for all $\bx\in D$. In particular, $\langle \bv,\bw\rangle_\bx=\bv^TM(\bx)\bw$ defines a point-dependent scalar product for all $\bx\in D$ and $\bv,\bw\in \mathbb R^n$.

In this section we show that the solution of a certain matrix-valued PDE is a contraction metric and gives information about the existence and uniqueness of a periodic orbit as well as its basin of attraction. 
There are different versions of the contraction condition in the literature; the one we present synchronizes the time between adjacent trajectories such that the difference vector $\bv$ is perpendicular on $ \bff(\bx)$, while the distance is measured with respect to the Riemannian metric $M$, i.e. $\bv^T M(\bx)\bv$. It can be generalized to synchronization perpendicular to $\bq(\bx)$, where $\bq(\bx)$ is not perpendicular to $\bff(\bx)$, however, the operator $L_M$ will take a different form, see \cite{leonov3}.

Other conditions  synchronize the time between adjacent trajectories such that the difference vector $\bv$ satisfies $\bv^T M(\bx) \bff(\bx)=0$, i.e. $\bv$ is perpendicular to $\bff(\bx)$  with respect to the metric $M$. This condition is less suitable for computations, as the unknown metric $M$ also appears in the condition for $\bv$.
The vector norm in the following theorem and the rest of the paper is the Euclidean norm $\|\cdot\|=\|\cdot\|_2$.

\begin{theorem}\label{1.1}
Let $K\subset \mathbb R^n$ be a compact, connected and positively invariant set which does not contain any equilibrium. Let $M\in C^1(K,\mathbb S^n)$ be a Riemannian metric and let  $L_M(\bx)\le -\nu<0$ for all $\bx\in K$ where
\begin{eqnarray*}
	L_M(\bx)&=&\max_{\bv\in\mathbb R^n, \bv^TM(\bx)\bv=1,\bv^T\bff(\bx)=0}L_M(\bx;\bv)\\
L_M(\bx;\bv)&=&\frac{1}{2}\bv^T\bigg(M'(\bx)+D\bff(\bx)^TM(\bx)+M(\bx)D\bff(\bx)\\
&&-\frac{M(\bx)\bff(\bx)\bff(\bx)^T(D\bff(\bx)+D\bff(\bx)^T)}{\|\bff(\bx)\|^2}\\
&&-\frac{(D\bff(\bx)+D\bff(\bx)^T)\bff(\bx)\bff(\bx)^TM(\bx)}{\|\bff(\bx)\|^2}\bigg)\bv,
\end{eqnarray*}		
and $(M'(\bx))_{i,j=1,\ldots,n}=(\nabla M_{ij}(\bx))^T\bff(\bx)$ is the matrix of the orbital derivatives of $M_{ij}$ along solutions of \eqref{ODE}.
	
Then there is one and only one periodic orbit $\Omega\subset K$, it is exponentially stable and the real part of all Floquet exponents apart from the trivial one is $\le -\nu$. Moreover, $K\subset A(\Omega)$ and $M$ is called a contraction metric.
\end{theorem}
For a sketch of the proof see   \ref{proof}.
	
	We intend to determine a matrix-valued function $M$ as above through a matrix-valued PDE.
	For $M\in C^1(\mathbb R^n,\mathbb S^n)$ and $\bx\in \mathbb R^n$ with $\bff(\bx)\not=\bnull$ define the first-order linear differential operator
	\begin{eqnarray}
		LM(\bx)&:=&	M'(\bx)+D\bff(\bx)^TM(\bx)+M(\bx)D\bff(\bx)\nonumber\\
		&&-\frac{M(\bx)\bff(\bx)\bff(\bx)^T(D\bff(\bx)+D\bff(\bx)^T)}{\|\bff(\bx)\|^2}\nonumber
		\\
		&&-\frac{(D\bff(\bx)+D\bff(\bx)^T)\bff(\bx)\bff(\bx)^TM(\bx)}{\|\bff(\bx)\|^2}.\label{LM}
	\end{eqnarray}
	For all $\bx\in \mathbb R^n$ with $\bff(\bx)\not=\bnull$ we also 
	define 
	\begin{eqnarray}
		P_\bx&:=&I-\frac{\bff(\bx) \bff(\bx)^T}{\|\bff(\bx)\|^2}.\label{P}
	\end{eqnarray}
	
	It is easy to see that $P_\bx $ is a  projection 
	onto the hyperplane perpendicular to $\bff(\bx)$, i.e.
	$P_\bx \bff(\bx)=\bnull$ 
	and $P_\bx P_\bx=P_\bx$. Moreover, we have $P_\bx \bv=\bv$ for all $\bv\in\mathbb R^n$  with $\bff(\bx)^T\bv=0$.

In the next proposition we will show that the solution of the matrix-valued PDE \eqref{PDE0}  is a contraction metric in the sense of Theorem \ref{1.1}. 
In Theorem \ref{main} we will show that if $M$ also satisfies an extra condition at one point \eqref{initial}, then we can conclude the positive definiteness of $M(\bx)$ for all $\bx\in A(\Omega)$.

\begin{proposition}
	Let $K\subset \mathbb R^n$ be a compact set which does not contain any equilibrium. 
	Let  $B\in C^{0}(K,\mathbb S^n)$ and $M\in C^{1}(K,\mathbb S^n)$ be such that both $B(\bx)$ and $M(\bx)$ are positive definite for each $\bx\in K$. Let $M$ satisfy 
  	\begin{eqnarray}
  		LM(\bx)&=&-P_\bx^TB(\bx)P_\bx\label{PDE0}
		\end{eqnarray}
		for all $\bx\in K$.
		
		Then there is are $\Lambda,\lambda>0$ such that $\bv^TB(\bx)\bv\ge \lambda \|\bv\|^2$ and $\bv^TM(\bx)\bv\le \Lambda \|\bv\|^2$ hold for all $\bx\in K$  and all $\bv\in \mathbb R^n$. Moreover, 
$$L_M(\bx)\le -\frac{\lambda}{2\Lambda}=:-\nu<0.$$
	\end{proposition}
	\begin{proof}
		The definition of $\lambda$ and $\Lambda$ follows from the fact that $B$ and $M$ are positive definite and continuous on the compact set $K$. 
		We have
\begin{eqnarray*}
2	L_M(\bx)&=&\max_{\bv\in\mathbb R^n, \bv^TM(\bx)\bv=1,\bv^T\bff(\bx)=0}
	\bv^T LM(\bx)\bv\\
	&=&-\max_{\bv\in\mathbb R^n, \bv^TM(\bx)\bv=1,\bv^T\bff(\bx)=0}
	\bv^T P_\bx^TB(\bx)P_\bx\bv\\
	&=&-\max_{\bv\in\mathbb R^n, \bv^TM(\bx)\bv=1,\bv^T\bff(\bx)=0}
	\bv^T B(\bx)\bv\\
	&\le&-\lambda\max_{\bv\in\mathbb R^n, \bv^TM(\bx)\bv=1,\bv^T\bff(\bx)=0}
	\|\bv\|^2\\
	&\le&-\frac{\lambda}{\Lambda}.
	\end{eqnarray*}
	\end{proof}
%
%
%
%


\section{Existence}
\label{ex}

Given an exponentially stable periodic orbit, we will now show the existence and uniqueness of the solution of \eqref{PDE} in its basin of attraction. We need to fix one value in \eqref{initial} to guarantee that $M$ is positive definite and to obtain uniqueness in Section \ref{uni}.

\begin{theorem}\label{main}
	Let $\Omega$ be an exponentially stable periodic orbit of $\dot{\bx}=\bff(\bx)$, $\bff\in C^\sigma(\mathbb R^n,\mathbb R^n)$, where $\sigma\ge 2$, with basin of attraction $A(\Omega)$. Fix $\bx_0\in A(\Omega)$ and $c_0\in\mathbb R^+$. 
	Let $B\in C^{\sigma-1}(A(\Omega),\mathbb S^n)$ be such that $B(\bx)$ is positive definite for all $\bx\in A(\Omega)$ and define $C\in C^{\sigma-1}(A(\Omega),\mathbb S^n)$
	by (see \eqref{P})
	$$C(\bx)=P_\bx^T B(\bx)P_\bx.$$
	
Then there exists a solution	$M\in C^{\sigma-1}(A(\Omega),\S^n)$ of the linear matrix-valued  PDE (see \eqref{LM})
	\begin{eqnarray}
		LM(\bx)&=&-C(\bx)\text{ for all }\bx\in A(\Omega)\label{PDE}\\
\text{ satisfying }
\bff(\bx_0)^TM(\bx_0)\bff(\bx_0)&=&c_0\|\bff(\bx_0)\|^4.\label{initial}
\end{eqnarray}

		The solution $M(\bx)$ is positive definite for all $\bx\in A(\Omega)$ and it is of the form 
		$$M(\bx)=\int_0^\infty  \Phi(t,0;\bx)^TC(S_t\bx)\Phi(t,0;\bx)\,dt+c_0 \bff(\bx)\bff(\bx)^T,$$
		where  $\Phi(t,0;\bx)$ denotes the principal fundamental matrix solution of  $\dot{\bphi}(t)=D(S_t\bx)\bphi(t)$ with $\Phi(0,0;\bx)=I$.
	\end{theorem}
\begin{proof}
	Denote \begin{eqnarray}
		M_1(\bx)&=&\int_0^\infty  \Phi(t,0;\bx)^TC(S_t\bx)\Phi(t,0;\bx)\,dt\label{defM}
		\end{eqnarray}
		 and $M_2(\bx)= \bff(\bx)\bff(\bx)^T$, so that $M(\bx)=M_1(\bx)+c_0M_2(\bx)$.
	It is clear that $M_2\in C^{\sigma}(A(\Omega),\mathbb S^n)$.

We will first show \eqref{initial} in Step 1.
In Step 2 we will show  $LM_2(\bx)=0$. In Step 3 we will prove estimates on $P_{S_t\bx} \Phi(t,0;\bx)$ which will then enable us to show  $LM_1(\bx)=-C(\bx)$ in Step 4, proving \eqref{PDE}.
In Step 5 we will show that  $M_1$ is well defined and $C^{\sigma-1}$. Finally, in Step 6, we show that $M$ is positive definite.

\vspace{0.3cm}

\noindent
\underline{\bf Step 1:	 $M$ satisfies \eqref{initial}}
\vspace{0.2cm}

To show \eqref{initial}, note that $\bff(S_t\bx)$ solves $\dot{\bphi}(t)=D\bff(S_t\bx)\bphi(t)$. Hence,
$\Phi(t,0; \bx)\bff(\bx)=\bff(S_t\bx)$.
This shows that for all $\bx\in A(\Omega)$
\begin{eqnarray*}
	\bff(\bx)^TM_1(\bx)\bff(\bx)&=&
	\int_0^\infty  \bff(S_t\bx)^TC(S_t\bx)\bff(S_t\bx)\,dt\\ 
	&=&
	 \int_0^\infty  \bff(S_t\bx)^TP_{S_t\bx}^TB(S_t\bx)P_{S_t\bx}\bff(S_t\bx)\,dt\ =\ 0
	\end{eqnarray*}
	since $P_{S_t\bx}\bff(S_t\bx)=\bnull$. On the other hand we have
	\begin{eqnarray*}
		\bff(\bx)^TM_2(\bx)\bff(\bx)&=&\|\bff(\bx)\|^4.
	\end{eqnarray*}
	This shows \eqref{initial}.

	\vspace{0.3cm}
	
	\noindent
	\underline{\bf Step 2:	 $LM_2(\bx)=0$}
	\vspace{0.2cm}
	
	We have, using $(\bff(\bx))'=D\bff(\bx)\bff(\bx)$,
\begin{eqnarray*}
	LM_2(\bx)&=&D\bff(\bx)\bff(\bx)\bff(\bx)^T+\bff(\bx)\bff(\bx)^TD\bff(\bx)^T\\
	&&+
	D\bff(\bx)^T\bff(\bx)\bff(\bx)^T+\bff(\bx)\bff(\bx)^TD\bff(\bx)\\
	&&
	-\frac{\bff(\bx)\bff(\bx)^T\bff(\bx)\bff(\bx)^T(D\bff(\bx)+D\bff(\bx)^T)}{\|\bff(\bx)\|^2}\\
		&&
		-\frac{(D\bff(\bx)+D\bff(\bx)^T)\bff(\bx)\bff(\bx)^T\bff(\bx)\bff(\bx)^T}{\|\bff(\bx)\|^2}\\
		&=&0.\end{eqnarray*}

	\vspace{0.3cm}
	
	\noindent
	\underline{\bf Step 3: $P_{S_t\bx}\Phi(t,0;\bx)$ decreases exponentially }
	\vspace{0.2cm}
	
To proceed with the proof, we will now show that $P_{S_t\bx}\Phi(t,0;\bx)$ decreases exponentially. This is done in several sub-steps. First we give an estimate for points $\bx=\bp\in\Omega$ on the periodic orbit in Lemma \ref{help1}.
Then we focus on points in a neighborhood $U$ of the periodic orbit in Lemma \ref{help}; this will imply the estimate for all points $\bx\in A(\Omega)$ in Step 4, see Lemma \ref{3.6}. The matrix norm in the following lemma and the rest of the paper is $\|\cdot\|=\|\cdot\|_2$, induced by the vector norm and sub-multiplicative.

	\begin{lemma}\label{help1}
	Let $-\nu$ be the largest real part of the non-trivial Floquet exponents of the periodic orbit $\Omega$ and 
let $\epsilon>0$. 

Then there is a constant $c_1>0$ such that for all   $\bp\in\Omega$ and all $0\le s\le t$ we have
	\begin{eqnarray}
	\|P_{S_t\bp} \Phi(t,0;\bp)\Phi(s,0;\bp)^{-1}\|&\le& c_1e^{(-\nu+\epsilon)(t-s)}\label{1.3first}\\
		\|\Phi(t,0;\bp)\Phi(s,0;\bp)^{-1})\|&\le& c_1 		\label{1.3second}\end{eqnarray}
		where $\Phi(t,0;\bp)$ is the principal fundamental matrix solution of 
%
the first variation equation
	 \begin{eqnarray}\dot{\bphi}(t)&=&D\bff(S_t\bp)\bphi(t)\label{variation}
	 	\end{eqnarray}
with		$\Phi(0,0;\bp)=I$.
	\end{lemma}
	\begin{proof} 
	Note that it is sufficient to prove the result for a fixed point $\bp\in \Omega$. Indeed, if $\bq=S_\theta\bp$ is a different point on the periodic orbit, then, see \eqref{eq0.3}
\begin{eqnarray*}
\Phi(t,0;S_\theta\bp)&=&\Phi(t+\theta,\theta;\bp)\\
	&=&\Phi(t+\theta,0;\bp)\Phi(\theta,0;\bp)^{-1}\\
\Phi(t,0;\bq)\Phi(s,0;\bq)^{-1}	&=&\Phi(t+\theta,0;\bp) 
\Phi(s+\theta,0;\bp)^{-1}
\end{eqnarray*}
and the result for $\bq$ follows from the result for $\bp$.

	Equation
	\eqref{variation}  
  is a $T$-periodic, linear equation for $\bphi$, where $T$ is the period of the periodic orbit $\Omega$, and $D\bff$ is $C^{\sigma-1}$.
By Floquet theory, the principal fundamental matrix solution $\Phi(t,0;\bp)$ of (\ref{variation}) with $\Phi(0,0;\bp)=I$   can be written as
$$\Phi(t,0;\bp)=Q(t)e^{Bt},$$
where $Q(\cdot)\in C^{\sigma-1}(\mathbb R,\mathbb C^{n\times n})$ is $T$-periodic with $Q(0)=Q(T)=I$, and $B\in\mathbb C^{n\times n}$. The eigenvalues of $B$ are $0$ with algebraic multiplicity one and the others have a real part $\le -\nu<0$.   Let $S\in\mathbb C^{n\times n}$ be an invertible matrix such that $S^{-1}BS=A$ is in a special Jordan Normal Form, where the complex eigenvalues are on the diagonal and 
the 1 on the side diagonal is replaced by $\epsilon$, and the first eigenvalue is $0$.

%
Let $\be_1,\ldots,\be_n\in \mathbb R^n$ denote the standard basis of $\mathbb R^n$. We have
\begin{eqnarray}\|e^{At}\bx\|&\le& e^{(-\nu+\epsilon)t}\|\bx\|\qquad \text{for all $\bx\in \spann(\be_2,\ldots,\be_n)$}\label{exp}\\
\text{and }
e^{At}\be_1&=&\be_1\text{ for all }t\ge 0.\label{1.5} 
\end{eqnarray}

Now we show that $\bff(S_t\bp)=\lambda Q(t)S\be_1$ holds for all $t\in \mathbb R$ with $\lambda\in\mathbb C\setminus \{0\}$. 
Indeed, since $\bff(S_t\bp)$ solves (\ref{variation}), we have for all $s\in\mathbb R$
\begin{eqnarray}
\bff(S_t\bp)&=&\Phi(t,0;\bp)\Phi(s,0;\bp)^{-1}\bff(S_s\bp)\nonumber\\
&=&Q(t)e^{B(t-s)}Q(s)^{-1}\bff(S_s\bp)\nonumber\\
&=&Q(t)Se^{A(t-s)}S^{-1}Q(s)^{-1}\bff(S_s\bp).\label{eqi}
\end{eqnarray}
For $t=s+T$ we have $\bff(S_s\bp)=\bff(S_t\bp) $ and $Q(s)=Q(t)$ by the periodicity and thus
\begin{eqnarray*}
S^{-1}Q(t)^{-1}\bff(S_t\bp)&=& e^{AT}S^{-1}Q(t)^{-1}\bff(S_t\bp).
\end{eqnarray*}
The form of $A$ implies that $S^{-1}Q(t)^{-1}\bff(S_t\bp)=\lambda \be_1$ with $\lambda\not=0$, and thus 
\begin{eqnarray}
\bff(S_t\bp)&=& \ \lambda Q(t)S\be_1\label{e1}
\end{eqnarray}
holds for all $t\in \mathbb R$.

We have
\begin{eqnarray*}
\Phi(t,0;\bp)\Phi(s,0;\bp)^{-1}
&=&Q(t)e^{B(t-s)}Q(s)^{-1}\\
\|\Phi(t,0;\bp)\Phi(s,0;\bp)^{-1}\|
&\le&\|Q(t)\| \|Q(s)^{-1}\| \|S\| \|S^{-1}\| \|e^{A(t-s)}\|\\
&\le&\max_{t'\in [0,T]}\|Q(t')\|
\max_{s'\in [0,T]} \|Q(s')^{-1}\| \|S\| \|S^{-1}\| 
\end{eqnarray*}
since $Q$ is $T$-periodic.
This shows \eqref{1.3second}.

Fix $s\ge 0$ and $\bc\in \mathbb R^n$. 
Let us write   
\begin{eqnarray*}\Phi(t,0;\bp)\Phi(s,0;\bp)^{-1}\bc
&=&Q(t)Se^{A(t-s)}S^{-1}Q(s)^{-1}\bc\\
&=&\sum_{i=1}^n\beta_iQ(t)Se^{A(t-s)}\be_i,
\end{eqnarray*}
where we have defined the $\beta_i\in \mathbb C$ by 
$\sum_{i=1}^n\beta_i\be_i=S^{-1}Q(s)^{-1}\bc$.  Note that $\sum_{i=1}^n|\beta_i|^2=\|S^{-1}Q(s)^{-1}\bc\|^2$.
Using  \eqref{P} and \eqref{1.5}, we have
\begin{eqnarray}
P_{S_t\bp} 
\Phi(t,0;\bp)\Phi(s,0;\bp)^{-1}\bc&=&
\sum_{i=2}^n\beta_i
Q(t)S e^{A(t-s)}\be_i+
\beta_1Q(t)S\be_1\nonumber\\
&&
-\sum_{i=2}^n\beta_i
\frac{\bff(S_t\bp)^TQ(t)Se^{A(t-s)}\be_i}{\|\bff(S_t\bp)\|^2}\bff(S_t\bp)\nonumber\\
&&
-\beta_1
\frac{\bff(S_t\bp)^TQ(t)S \be_1}{\|\bff(S_t\bp)\|^2}\bff(S_t\bp)\nonumber\\
&=&
\sum_{i=2}^n\beta_i
Q(t)S e^{A(t-s)}\be_i \nonumber\\
&&
-\sum_{i=2}^n\beta_i
\frac{\bff(S_t\bp)^TQ(t)Se^{A(t-s)}\be_i}{\|\bff(S_t\bp)\|^2}\bff(S_t\bp).\label{e_z1}
\end{eqnarray}
The two terms with $\beta_1$ cancel each other out 
since by \eqref{e1}
$$\frac{\bff(S_t\bp)^T Q(t)S\be_1}{\|\bff(S_t\bp)\|^2}
\bff(S_t\bp)=\frac{\bff(S_t\bp)^T\lambda Q(t)S\be_1}{\|\bff(S_t\bp)\|^2}
Q(t)S\be_1=Q(t)S\be_1.$$

%
%
%
In particular, at $t=s$, we have with \eqref{e1}
\begin{eqnarray}
P_{S_t\bp}  \bc &=&
\sum_{i=2}^n
\beta_i Q(t)S \be_i
-\sum_{j=2}^n\beta_j  
\frac{\bff( S_t\bp)^T Q(t)S \be_j}{\|\bff(S_t \bp)\|^2} \bff( S_t\bp)\nonumber\\
&=&Q(t)S\left[
\sum_{i=2}^n
\beta_i  \be_i
- \lambda
\frac{\bff(S_t \bp)^T \left(\sum_{j=2}^n\beta_jQ(t)S\be_j\right)}{\|\bff( S_t\bp)\|^2} \be_1\right]\nonumber\\
\|S^{-1}Q(t)^{-1}P_{S_t \bp} \bc\|^2&\ge&\sum_{i=2}^n|\beta_i|^2.\label{extra}
\end{eqnarray}
We have from \eqref{e_z1} for $0\le s \le t$
  \begin{eqnarray*}\lefteqn{
\|
P_{S_t\bp} 
\Phi(t,0;\bp)\Phi(s,0;\bp)^{-1}\bc\|}\\
&\le& \|Q(t)\| \ \|S\|
\left\|\sum_{i=2}^n\beta_i\be_i\right\| e^{(-\nu+\epsilon)(t-s)}\\
&&
+\frac{\|\bff(S_t\bp)\|^2}{\|\bff(S_t\bp)\|^2}\|Q(t)\| \ \|S\|\left\|\sum_{j=2}^n\beta_j\be_j\right\| e^{(-\nu+\epsilon)(t-s)}
\\
&\le&2\,\|Q(t)\| \ \|S\|\|
\sqrt{\sum_{i=2}^n|\beta_i|^2} \,e^{(-\nu+\epsilon)(t-s)}\\
&\le&2\,\max_{t'\in [0,T]}\|Q(t')\|  \,\|S\| \,\max_{s'\in [0,T]}\|Q(s')^{-1}\| \, \|S^{-1}\|\,\| \bc\|\, 
e^{(-\nu+\epsilon)(t-s)}
\end{eqnarray*}
by \eqref{extra}.
This shows \eqref{1.3first} and the lemma.
\end{proof}
%
%
		
We use the following result from \cite[Corollary 3.6]{other}. In a neighborhood $U$ of the periodic orbit we define a projection of a point $\bx\in U$ onto a point $\pi(\bx)\in \Omega$ on the periodic orbit. We can synchronize the times of trajectories through $\bx$ (time $t$) and $\pi(\bx)=\bp$ (time $\theta$) such that $\pi(S_t\bx)=S_{\theta_\bx(t)}\bp$. Moreover, we define a distance  of $S_t \bx$ to the periodic orbit, in particular to $\pi(S_t\bx)$, which exponentially decreases along solutions.
This notion of stability is also called Zhukovskii stability and its relation to Lyapunov stability has been studied, e.g. in \cite{leonov2}.

\begin{lemma}\label{help_other}Let $\Omega$ be an exponentially stable periodic orbit of $\dot{\bx}=\bff(\bx)$ with $\bff\in C^\sigma (\mathbb R^n,\mathbb R^n)$ and $\sigma\ge 2$ and denote by $-\nu<0$ the maximal real part of all its non-trivial Floquet exponents.

For $\epsilon\in (0,\min(\nu,1))$ there is a compact, positively invariant neighborhood $U$ of $\Omega$ with $\Omega\subset U^\circ$ and $U\subset A(\Omega)$ and a map $\pi\in C^{\sigma-1}( U, \Omega)$ with $\pi(\bx)=\bx$ if and only if $\bx\in\Omega$.
	Furthermore, for a fixed $\bx\in U$ there is a bijective $C^{\sigma-1}$ map
	$\theta_\bx\colon [0,\infty)\to[0,\infty)$ with inverse
		$t_\bx=\theta_\bx^{-1}\in C^{\sigma-1}([0,\infty),[0,\infty))$
	such that $\theta_\bx(0)=0$ and 
	$$\pi(S_t\bx)=S_{\theta_\bx(t)}\pi(\bx)$$
	for all $t\in [0,\infty)$.
	Moreover, $\dot{\theta}_\bx(t)\in \left[1-\epsilon,1+\epsilon\right]$ for all $t\ge 0$ and 
	 $\dot{t}_\bx(\theta)\in \left[1-\epsilon,1+\epsilon\right]$ for all $\theta\ge 0$.
	 
	 Finally,  there is a constant $C>0$ such that 
		\begin{eqnarray}|\dot{t}_\bx(\theta)-1|&\le& Ce^{-\mu_0\theta}\label{res1}\\
	\|S_{t_\bx(\theta)}\bx-S_\theta \pi(\bx)\|&\le& Ce^{-\mu_0 \theta}\|\bx-\pi(\bx)\|\label{res2}
	\end{eqnarray}
	for all $\theta\ge 0$ and all $\bx\in U$, where $\mu_0=\nu-\epsilon$.
	\end{lemma}

Using Lemma \ref{help_other}, we will now show Lemma \ref{help}.
	
\begin{lemma}\label{help}
Using the notation of Lemma \ref{help_other} with $0<\epsilon<\min(1,\nu/2)$, 
			there are constants $C>0$ and
			$\kappa=\frac{\nu-2\epsilon}{1+\epsilon}>0$ such that for all $\bx\in U$  we have
			\begin{eqnarray}
				\|P_{S_t\bx} \bphi(t)\|&\le& Ce^{-\kappa t}\|P_{\bx} \bphi(0)\|\label{help.1}\\
				\|\bphi(t)\|&\le& C \|\bphi(0)\|\label{help.2}
			\end{eqnarray}
		 for all $t\ge 0$. Here, $\bphi(t)$ is a solution of the first variation equation
			\begin{eqnarray}
				\dot{\bphi}(t)&=&D\bff(S_t\bx)\bphi(t).\label{var}
				\end{eqnarray}
		\end{lemma}

\begin{proof} Denote $\mu_0=\nu-\epsilon>0$, let $\bx\in U$
 and denote the synchronized time by $\theta_{\bx}(t)=\theta(t)$, see Lemma \ref{help_other}. We now drop the subscript.

Let   
$A(\theta):=D\bff(S_\theta\bp)$ with $\bp=\pi(\bx)\in \Omega$, where $\pi$ was defined in Lemma \ref{help_other}. 
Using the inverse of $\theta(t)$, namely $t=\theta^{-1}$, we define
$D(\theta):= D\bff(S_{t(\theta)}\bx)$ and $\bpsi(\theta):=\tpsi(t(\theta))$.
Then we have by \eqref{var}
\begin{eqnarray}\frac{d}{d\theta}\bpsi(\theta)&=&\frac{d}{dt}\tpsi(t(\theta))\cdot \dot{t}(\theta) 
\ =\ D(\theta)\bpsi(\theta)\dot{t}(\theta).\label{psi-theta}
\end{eqnarray}

Since $A(\theta)=D\bff(S_\theta \bp)$ is $T$-periodic, we can use  Floquet Theory to express the principal fundamental matrix solution $\Phi(\theta,0;\bp)$ of $\dot{\by}(\theta)=A(\theta)\by(\theta)$  as in Lemma \ref{help1}. In the following   we will abbreviate it by $\Phi(\theta)$, where $\Phi(0)=I$.

As $\Phi(\theta)$ exists and is non-singular for all $\theta\in\mathbb R_0^+$, we have
\begin{eqnarray}
0&=&\frac{d}{d\theta}\left(\Phi(\theta)\Phi(\theta)^{-1}\right)\nonumber\\
&=&\left(\frac{d}{d\theta}\Phi(\theta)\right)\Phi(\theta)^{-1}+\Phi(\theta)\left(\frac{d}{d\theta}\Phi(\theta)^{-1}\right)\nonumber\\
\frac{d}{d\theta}\Phi(\theta)^{-1}&=&-\Phi(\theta)^{-1} \left(\frac{d}{d\theta}\Phi(\theta)\right)\Phi(\theta)^{-1}\nonumber\\
&=&-\Phi(\theta)^{-1}A(\theta).\label{phi-1}
\end{eqnarray}

Using \eqref{psi-theta} and \eqref{phi-1} we have
\begin{eqnarray*}
\frac{d}{d\theta}\left(\Phi(\theta)^{-1}\bpsi(\theta)\right)&=&
-\Phi(\theta)^{-1}A(\theta)\bpsi(\theta)+\Phi(\theta)^{-1}D(\theta)\bpsi(\theta)\dot{t}(\theta)\\
&=&
\Phi(\theta)^{-1}\left(D(\theta)-A(\theta)+D(\theta)(\dot{t}(\theta)-1)\right)\bpsi(\theta).
\end{eqnarray*}
Integrating both sides from $0$ to $\theta\ge 0$ we obtain
\begin{eqnarray}
\lefteqn{
\Phi(\theta)^{-1}\bpsi(\theta)-\bpsi(0)}\nonumber\\&=&\int_0^\theta \Phi(s)^{-1}\left(D(s)-A(s)+D(s)(\dot{t}(s)-1)\right)\bpsi(s)\,ds\nonumber\\
\bpsi(\theta)&=&\Phi(\theta)\bpsi(0)\nonumber\\
&&+\int_0^\theta \Phi(\theta) \Phi(s)^{-1} \left(D(s)-A(s)+D(s)(\dot{t}(s)-1)\right)\bpsi(s)\,ds.\label{psi11}
\end{eqnarray}

Since $D\bff$ is $C^1$ on the compact set  $U$, there is a Lipschitz constant $L>0$ such that
\begin{eqnarray*}
\|D(s)-A(s)\|&=&\|D\bff(S_{t(s)}\bx)-D\bff(S_s\bp)\|\\
&\le&L\|S_{t(s)}\bx-S_s \bp\|\\
&\le&LCe^{-\mu_0 s}\|\bx-\bp\|
\end{eqnarray*}
by \eqref{res2}.
Hence, altogether we have with \eqref{res1} and using that $D(s)=D\bff(S_{t(s)}\bx)$ is bounded for all $s\in [0,\infty)$ and $\bx\in U$
\begin{eqnarray}
\left\|D(s)-A(s)+D(s)(\dot{t}(s)-1)\right\|
&\le&d_1e^{-\mu_0 s}.\label{D-A}
\end{eqnarray}
\vspace{0.3cm}
\noindent
\underline{\bf Estimate on $\|\bpsi(\theta)\|$}

From \eqref{psi11} we obtain
\begin{eqnarray}
\|\bpsi(\theta)\|&\le&\|\Phi(\theta)\| \|\bpsi(0)\|+\int_0^\theta \|\Phi(\theta) \Phi(s)^{-1}\| \cdot\nonumber\\
&&
\left\|D(s)-A(s) +D(s) (\dot{t}(s)-1)\right\|\cdot \|\bpsi(s)\| \,ds\label{zwischen1}
\end{eqnarray}
for all $\theta\ge 0$.
We have
\begin{eqnarray}
\|\Phi(\theta)\|&\le& c_1\text{ for }\theta\ge 0\label{h1}\\
\|\Phi(\theta) \Phi(s)^{-1} \|&\le&c_1\text{ for }\theta-s\ge 0, \label{h2}
\end{eqnarray}
see Lemma \ref{help1}. 

Using these estimates in \eqref{zwischen1}, as well as \eqref{D-A}
gives
\begin{eqnarray*}
\|\bpsi(\theta)\|&\le&c_1\|\bpsi(0)\|+\int_0^\theta c_1 d_1e^{-\mu_0 s} \|\bpsi(s)\| \,ds.
\end{eqnarray*}
Now we apply Lemma \ref{gronwall} with $r(\theta)=\|\bpsi(\theta)\|$, $a(\theta)=c_1\|\bpsi(0)\|$,
$K(\theta)=d_1c_1 $ and $b(\theta)=e^{-\mu_0 \theta}$, giving
\begin{eqnarray*}
\|\bpsi(\theta)\|
&\le&c_1\|\bpsi(0)\|+d_1c_1^2\|\bpsi(0)\|\int_0^\theta e^{-\mu_0 s} \,ds
\cdot \exp\left(\int_0^\theta d_1c_1 e^{-\mu_0 s}\,ds\right)\nonumber\\
&\le&c_1\|\bpsi(0)\|+\frac{d_1c_1^2}{\mu_0}\|\bpsi(0)\|
\cdot \exp\left(\frac{ d_1c_1 }{\mu_0 }\right) 
\end{eqnarray*}
using $\int_0^\theta  e^{-\mu_0 s}\, ds=\frac{1}{\mu_0} (1-e^{-\mu_0 \theta})\le \frac{1}{\mu_0}$. Note that this holds for all $\theta\ge 0$ since $\bpsi(\theta)$ is continuous.
Using $\tpsi(t(\theta))=\bpsi(\theta)$ and that $t(\theta)$ is bijective this shows \eqref{help.2}.

\vspace{0.3cm}
\noindent
\underline{\bf Estimate on $\|P_{S_{t(\theta)\bx}}\bpsi(\theta)\|$}
\vspace{0.2cm}

Note that by \eqref{P}
we have
\begin{eqnarray*}
\frac{d}{dt}P_{S_t \bx}
&=&-\frac{D\bff(S_t\bx)\bff(S_t\bx)\bff(S_t\bx)^T+\bff(S_t\bx)\bff(S_t\bx)^TD\bff(S_t\bx)^T}{\|\bff(S_t\bx)\|^2}\\
&&+
\frac{\bff(S_t\bx)\bff(S_t\bx)^T}{\|\bff(S_t\bx)\|^4}
\bff(S_t\bx)^T(D\bff(S_t\bx)^T+D\bff(S_t\bx))\bff(S_t\bx).
\end{eqnarray*}
Hence, using  \eqref{psi-theta} we have
\begin{eqnarray}
\lefteqn{
\frac{d}{d\theta}(P_{S_{t(\theta)}\bx}\bpsi(\theta))}\nonumber\\
&=& \dot{t}(\theta) 
\bigg(-\frac{D\bff(S_{t(\theta)}\bx)\bff(S_{t(\theta)}\bx)\bff(S_{t(\theta)}\bx)^T+\bff(S_{t(\theta)}\bx)\bff(S_{t(\theta)}\bx)^TD\bff(S_{t(\theta)}\bx)^T}{\|\bff(S_{t(\theta)}\bx)\|^2}\nonumber\\
&&+
\frac{\bff(S_{t(\theta)}\bx)\bff(S_{t(\theta)}\bx)^T}{\|\bff(S_{t(\theta)}\bx)\|^4}
\bff(S_{t(\theta)}\bx)^T(D\bff(S_{t(\theta)}\bx)^T+D\bff(S_{t(\theta)}\bx))\bff(S_{t(\theta)}\bx)\nonumber\\
&&+D\bff(S_{t(\theta)}\bx)-\frac{\bff(S_{t(\theta)}\bx)\bff(S_{t(\theta)}\bx)^T}{\|\bff(S_{t(\theta)}\bx)\|^2}
D\bff(S_{t(\theta)}\bx)\bigg)\bpsi(\theta)\nonumber\\
&=& \dot{t}(\theta) \left(
D\bff(S_{t(\theta)}\bx)-\bff(S_{t(\theta)}\bx)\br(\theta)^T\right)P_{S_{t(\theta)}\bx}\bpsi(\theta)\label{psi-r}
\end{eqnarray}
where
\begin{eqnarray*}
\br(\theta)^T&=&\frac{\bff(S_{t(\theta)}\bx)^T
(D\bff(S_{t(\theta)}\bx)^T+D\bff(S_{t(\theta)}\bx))}{\|\bff(S_{t(\theta)}\bx)\|^2}.
\end{eqnarray*}
Note that there is a constant $R>0$ such that for all $\bx\in U$ and all $\theta\ge 0$
\begin{eqnarray}
\|\br(\theta)\|&\le&R\label{R}
\end{eqnarray}
as $\bff\in C^1$ in the compact, positively invariant set $U$.
Using \eqref{phi-1} we have 
\begin{eqnarray*}
\frac{d}{d\theta}\left(\Phi(\theta)^{-1}P_{S_{t(\theta)} \bx}\bpsi(\theta)\right)
&=&
\Phi(\theta)^{-1}\left(D(\theta)-A(\theta)+D(\theta)(\dot{t}(\theta)-1)\right)P_{S_{t(\theta)}\bx}\bpsi(\theta)\\
&&-\Phi(\theta)^{-1} \bff(S_{t(\theta)}\bx)\br(\theta)^TP_{S_{t(\theta)}\bx}\bpsi(\theta)\dot{t}(\theta).
\end{eqnarray*}
Integrating both sides from $0$ to $\theta\ge 0$ we obtain
\begin{eqnarray}
\lefteqn{
\Phi(\theta)^{-1}P_{S_{t(\theta)} \bx}\bpsi(\theta)-P_{ \bx}\bpsi(0)}\nonumber\\&=&\int_0^\theta \Phi(s)^{-1}\left[D(s)-A(s)+D(s)(\dot{t}(s)-1)\right]P_{S_{t(s)}\bx}\bpsi(s)\,ds\nonumber\\
&&-\int_0^\theta \Phi(s)^{-1} \bff(S_{t(s)}\bx)\br(s)^TP_{ S_{t(s)}\bx}\bpsi(s)\dot{t}(s)\,ds\nonumber\\
P_{S_{t(\theta)} \bx}\bpsi(\theta)&=&\Phi(\theta)P_{ \bx}\bpsi(0)\nonumber\\
&&+\int_0^\theta \Phi(\theta) \Phi(s)^{-1} \left[D(s)-A(s)+D(s)(\dot{t}(s)-1)\right]P_{ S_{t(s)}\bx}\bpsi(s)\,ds\nonumber\\
&&-\int_0^\theta \Phi(\theta)\Phi(s)^{-1} \bff(S_{t(s)}\bx)\br(s)^TP_{ S_{t(s)}\bx}\bpsi(s)\dot{t}(s)\,ds.
\label{psi}
\end{eqnarray}

We now multiply with $P_{S_{t(\theta)}\bx}$ from the left, noting that $P $ is a projection.
\begin{eqnarray}
P_{S_{t(\theta) \bx}}\bpsi(\theta)&=&
P_{S_{t(\theta)}\bx}\Phi(\theta)P_{ \bx}\bpsi(0)\nonumber\\
&&+\int_0^\theta P_{S_{t(\theta)}\bx}\Phi(\theta) \Phi(s)^{-1} \left[D(s)-A(s)+D(s)(\dot{t}(s)-1)\right]P_{ S_{t(s)}\bx}\bpsi(s)\,ds\nonumber\\
&&-\int_0^\theta P_{S_{t(\theta)}\bx}\Phi(\theta)\Phi(s)^{-1} \bff(S_{t(s)\bx})\br(s)^TP_{ S_{t(s)}\bx}\bpsi(s)\dot{t}(s)\,ds.\label{zwischen11}
\end{eqnarray}

Let us focus on the term $P_{S_{t(\theta)}\bx}\Phi(\theta)\Phi(s)^{-1} \bff(S_{t(s)\bx})$.
Define $\by(\tau):=\bff(S_{t(\tau)}\bx)$. We have $\frac{d\by}{d\tau}(\tau)=D\bff(S_{t(\tau)}\bx)\by(\tau)\dot{t}(\tau)=D(\tau)\by(\tau)\dot{t}(\tau)$.  Hence, 
\begin{eqnarray*}
	\frac{d}{d\tau}\left(\Phi(\tau)^{-1}\by(\tau)\right)&=&\Phi(\tau)^{-1}[D(\tau)-A(\tau)+D(\tau)(\dot{t}(\tau)-1)]\by(\tau)\\
\Phi(\theta)^{-1}\by(\theta)-\Phi(s)^{-1}\by(s)&=&
\int_s^\theta \Phi(\tau)^{-1}[D(\tau)-A(\tau)+D(\tau)(\dot{t}(\tau)-1)]\by(\tau)\,d\tau\\
\by(\theta)&=&\Phi(\theta)\Phi(s)^{-1}\by(s)\\
&&+
\int_s^\theta \Phi(\theta) \Phi(\tau)^{-1}[D(\tau)-A(\tau)+D(\tau)(\dot{t}(\tau)-1)]\by(\tau)\,d\tau.
\end{eqnarray*}
Applying $P_{S_{t(\theta)}\bx}$ from the left and noting that   $	P_{S_{t(\theta)}\bx}\by(\theta)=	P_{S_{t(\theta)}\bx}
\bff(S_{t(\theta)}\bx)
=\bnull$
 we have
\begin{eqnarray}
\lefteqn{-
	P_{S_{t(\theta)}\bx}\Phi(\theta)\Phi(s)^{-1} \bff(S_{t(s)}\bx)}\nonumber\\
	&=&P_{S_{t(\theta)}\bx}
	\int_s^\theta \Phi(\theta) \Phi(\tau)^{-1}[D(\tau)-A(\tau)+D(\tau)(\dot{t}(\tau)-1)]\bff(S_{t(\tau)}\bx)\,d\tau.\qquad\ \label{help3}
	\end{eqnarray}
	
	We  have from Lemma \ref{help_other} $$\|
	P_{S_{t(\theta)} \bx}-P_{S_\theta\bp}\|\le L \|
	S_{t(\theta)} \bx-S_{\theta}\bp\|
	\le LC e^{-\mu_0 \theta}$$
	 since $P_\bx$ is continuously differentiable with respect to $\bx$ and hence Lipschitz continuous in the compact set $U$.
	
	Define $\rho_0=\nu-2\epsilon$ such that $0<\rho_0<\mu_0$.
	With $	\|P_{S_\theta\bp}\Phi(\theta)\|\le c_1e^{-\mu_0 \theta}$  for $\theta\ge 0$ from Lemma \ref{help1} 
	\begin{eqnarray}
		\|P_{S_{t(\theta)}\bx}\Phi(\theta)\|&\le&
		\|P_{S_{t(\theta)}\bx}-P_{S_\theta\bp}\| \cdot\|\Phi(\theta)\|+
		\|P_{S_\theta\bp}\Phi(\theta)\|\nonumber\\
		&\le&c_1 LC e^{-\mu_0 \theta}+c_1e^{-\mu_0 \theta}\text{ by }\eqref{h1}\nonumber\\
		&\le&c_3e^{-\rho_0 \theta} \text{ for }\theta\ge 0.\label{*0}
	\end{eqnarray} 
We also have for all $\theta\ge \tau\ge 0$ 
	\begin{eqnarray}
\|	P_{S_{t(\theta)}\bx} \Phi(\theta) \Phi(\tau)^{-1}\|
&	\le& \|	P_{S_{\theta}\bp} \Phi(\theta) \Phi(\tau)^{-1}\|+\|P_{S_{t(\theta)}\bx}-P_{S_{\theta}\bp}\|
\cdot \| \Phi(\theta) \Phi(\tau)^{-1}\|\nonumber\\
&\le&c_1e^{-\mu_0(\theta-\tau)}+c_1LC e^{-\mu_0 \theta}\text{ by Lemma \ref{help1} and }\eqref{h2}\nonumber\\
&\le&c_4e^{-\rho_0(\theta-\tau)}.\label{*}
\end{eqnarray}
 Hence,  using that $\|\bff(\bx)\|\le F$ for all $\bx\in U$, we obtain with \eqref{help3} and \eqref{D-A}
	\begin{eqnarray*}
	\|	P_{S_{t(\theta)}\bx}\Phi(\theta)\Phi(s)^{-1} \bff(S_{t(s)}\bx)\|
		&\le&
		\int_s^\theta c_4 e^{-\rho_0(\theta-\tau)}d_1e^{-\mu_0\tau}F\,d\tau
	\end{eqnarray*}
and hence, using \eqref{R}, $|\dot{t}(s)|\le 1+\epsilon$ and $\rho_0<\mu_0$
\begin{eqnarray*}
	\lefteqn{
	\left\|\int_0^\theta P_{S_{t(\theta)}\bx}\Phi(\theta)\Phi(s)^{-1} \bff(S_{t(s)\bx})\br(s)^TP_{ S_{t(s)}\bx}\bpsi(s)\dot{t}(s)\,ds\right\|}\\
	&\le&(1+\epsilon)Rc_4d_1F\int_0^\theta \int_s^\theta 
	\|P_{ S_{t(s)}\bx}\bpsi(s)\| e^{-\rho_0(\theta-\tau)}e^{-\mu_0\tau}\,d\tau\,ds\\
		&=&(1+\epsilon)Rc_4d_1F\int_0^\theta	\|P_{ S_{t(s)}\bx}\bpsi(s)\|e^{-\rho_0\theta} \int_s^\theta 
	 e^{-(\mu_0-\rho_0)\tau}\,d\tau\,ds\\
	 	&\le&(1+\epsilon)Rc_4d_1F\int_0^\theta	\|P_{ S_{t(s)}\bx}\bpsi(s)\|e^{-\rho_0\theta} 
	 	\frac{1}{\mu_0-\rho_0}e^{-(\mu_0-\rho_0)s}\,ds.
	\end{eqnarray*}

Hence, we have with \eqref{zwischen11},  \eqref{*0}, \eqref{*} and \eqref{D-A}
\begin{eqnarray*}
\|
P_{S_{t(\theta)} \bx}\bpsi(\theta)\|&\le&
c_3e^{-\rho_0 \theta}\|P_{ \bx}\bpsi(0)\|\nonumber\\
&&+\underbrace{\left(1+	(1+\epsilon)\frac{RF}{\mu_0-\rho_0}\right)c_4d_1}_{=:c_5} \int_0^\theta e^{-\rho_0(\theta-s)}e^{-\mu_0 s} \|P_{ S_{t(s)}\bx}\bpsi(s)\|\,ds.
\end{eqnarray*}

Lemma \ref{gronwall} with $r(\theta)=\|P_{S_{t(\theta)}\bx}\bpsi(\theta)\|$, 
	$a(\theta)=c_3e^{-\rho_0 \theta}\|P_{\bx}\bpsi(0)\|$,
$K(\theta)=c_5 e^{-\rho_0 \theta}$ and $b(\theta)=e^{\theta(\rho_0-\mu_0)}$ gives
\begin{eqnarray}\lefteqn{
\|P_{S_{t(\theta)}\bx}\bpsi(\theta)\|}\nonumber\\
&\le&c_3e^{-\rho_0 \theta}\|P_{\bx}\bpsi(0)\|\nonumber\\
&&+c_5e^{-\rho_0 \theta}\int_0^\theta c_3e^{-\rho_0 s}\|P_{\bx}\bpsi(0)\| e^{s(\rho_0-\mu_0)}\,ds
\cdot \exp\left(\int_0^\theta c_5 e^{-\rho_0 s}e^{s(\rho_0-\mu_0)}\,ds\right)\nonumber\\
&=&c_3e^{-\rho_0 \theta}\|P_{\bx}\bpsi(0)\|+c_3c_5\|P_{\bx}\bpsi(0)\| e^{-\rho_0 \theta}\int_0^\theta  e^{-\mu_0 s}\,ds
\cdot \exp\left(c_5\int_0^\theta e^{-\mu_0 s}\,ds\right)\nonumber\\
&\le&c_3e^{-\rho_0 \theta}\|P_{\bx}\bpsi(0)\|+\frac{c_3c_5}{\mu_0}\|P_{\bx}\bpsi(0)\| e^{-\rho_0 \theta}\cdot \exp\left(\frac{c_5}{\mu_0}\right)\nonumber
\end{eqnarray}
using $\int_0^\theta  e^{-\mu_0 s}\, ds=\frac{1}{\mu_0} (1-e^{-\mu_0 \theta})\le \frac{1}{\mu_0}$. Note that this holds for all $\theta\ge 0$ since $\bpsi(\theta)$ is continuous.

This proves \eqref{help.1} with $\kappa:=\rho_0/(1+\epsilon)$ using $\bphi(t(\theta))=\bpsi(\theta)$ and that $t$ is a bijection on $[0,\infty)$ as well as that we have for $\theta\ge 0$
$$t(\theta)=\int_0^\theta \dot{t}(s)\,ds\le (1+\epsilon) \theta,$$ using $\dot{t}(s)\le 1+\epsilon$ and  $t(0)=0$.
\end{proof}

\vspace{0.3cm}

\noindent
\underline{\bf Step 4: $LM_1(\bx)=-C(\bx)$}
\vspace{0.2cm}

Now fix $\bx\in A(\Omega)$. Denote by $\Phi(\tau,\theta;\bx)=\Phi(\tau,0;\bx)\Phi(\theta,0;\bx)^{-1}$ the state transition matrix.
Note that for fixed $\bx$ there exists a $\theta_0>0$ such that $S_\tau\bx$, $D\bff(S_\tau\bx)$  and thus also $\Phi(\tau,\theta;\bx)$ are defined for all $\tau,\theta\ge -\theta_0$, and $\Phi(\tau,\theta;\bx)$ is $C^{\sigma-1}$ with respect to $\bx$, $\tau$ and $\theta$. 

By the Chapman-Kolmogorov identities, cf. e.g. \cite{chicone}, p. 151, we have
\begin{eqnarray}
	\frac{d}{d\tau}\Phi(\tau,\theta;\bx)&=&D\bff(S_\tau\bx)\Phi(\tau,\theta;\bx),\nonumber\\
	\frac{d}{d\theta}\Phi(\tau,\theta;\bx)&=&-\Phi(\tau,\theta;\bx)D\bff(S_{\theta}\bx),\label{t0}\\
	\Phi(\theta,\theta;\bx)&=&I,\label{I}\\
	\Phi(\tau,0;S_\theta \bx)&=&\Phi(\tau+\theta,\theta;\bx).\label{eq0.3}
\end{eqnarray}
for all $\tau,\tau+\theta\ge -\theta_0$. Also,
\begin{eqnarray}
\Phi(\tau-T_0,-T_0;S_{\theta+T_0}\bx)&=&\Phi(\tau,0;S_\theta \bx)\label{eq0.4}
\end{eqnarray}
for $\tau\ge T_0\ge 0$ and $|\theta|\le \theta_0$.
The last two equations follow from the fact that both functions satisfy the same initial value problem. For example, both sides of \eqref{eq0.4} satisfy $\frac{d}{d\tau} \by(\tau)=D\bff(S_{\tau+\theta}\bx)\by(\tau)$.

Define
\begin{eqnarray}
g_T(\theta,\bx)&=&
\int_\theta^{T+\theta} \Phi(\tau,\theta; \bx)^T P_{S_\tau \bx}^T
B(S_\tau \bx)P_{S_\tau \bx}\Phi(\tau,\theta; \bx)\,d\tau.\label{g1}
\end{eqnarray}
We have for all $\theta\ge -\theta_0$ by a change of variables and (\ref{eq0.3})
\begin{eqnarray}
g_T(\theta,\bx)&=&
\int_0^{T} \Phi(\tau+\theta,\theta; \bx)^T P_{S_{\tau+\theta} \bx}^T
B(S_{\tau +\theta} \bx)P_{S_{\tau+\theta} \bx}\Phi(\tau+\theta,\theta; \bx)\,d\tau\label{gextra}
\\
&=&
\int_0^{T} \Phi(\tau,0; S_\theta\bx)^T P_{S_{\tau+\theta} \bx}^T
B(S_{\tau+\theta} \bx)P_{S_{\tau+\theta} \bx}\Phi(\tau,0; S_\theta\bx)\,d\tau.\label{g2}
\end{eqnarray}

We will show that $g_T(\theta,\bx)$ converges pointwise and $\frac{d }{d \theta}g_T(\theta,\bx)$ converges uniformly
in $|\theta|\le \theta_0$
as $T\to\infty$ so that $\frac{d }{d \theta}\lim_{T\to\infty}g_T(\theta,\bx)=\lim_{T\to \infty}\frac{d }{d \theta}g_T(\theta,\bx)$ for $|\theta|<\theta_0$. 
As the set $S=\overline{\bigcup_{t=-\theta_0}^\infty \{S_t\bx\}}$ is compact and  $B(\cdot)$ is continuous, there exists $B^*>0$ such that
\begin{eqnarray}
\|B(S_t\bx)\|&\le&B^*\label{esti}
\end{eqnarray}
for all $t\ge -\theta_0$.

\begin{lemma}\label{3.6}
	For a fixed $\bx\in A(\Omega)$ there exists $c>0$ such that
		\begin{eqnarray}  \|\Phi(\tau,0; S_\theta\bx)^T P_{S_{\tau}(S_\theta\bx)}^T
		B(S_{\tau+\theta}\bx)P_{S_{\tau}(S_\theta\bx)}\Phi(\tau,0; S_\theta\bx)\|
		&\le& ce^{-2\kappa \tau}\label{expab}
	\end{eqnarray}
	for all $|\theta|\le \theta_0$ and all $\tau\ge 0$.
	\end{lemma}
	\begin{proof}Since $\bx\in A(\Omega)$ and $\Omega\subset U^\circ$, where $U$ is compact and positively invariant, there exists $T_0$ such that
		$S_{\tau+\theta}\bx \in U$ for all $\tau\ge T_0$ and all $|\theta|\le \theta_0$.
		Since all terms in \eqref{expab} depend continuously on $\tau $ and $\theta$, we can choose $c$ such that the inequality holds for all $|\theta|\le \theta_0$ and all $\tau \in [0,T_0]$.

	We denote $\by=S_{\theta+T_0}\bx\in U$. With $t=\tau-T_0$ we have
		\begin{eqnarray*}
		\|P_{S_t\by}\bphi(t)\|	&\le&Ce^{-\kappa t}\|P_\by\bphi(0)\|
		\end{eqnarray*}
		by \eqref{help.1} of Lemma \ref{help}, where $\bphi(t)$ solves $\dot{\bphi}=D\bff(S_t\by)\bphi$. Note that by \eqref{eq0.4}  $\Phi(\tau,0;S_\theta \bx)=\Phi(\tau-T_0,-T_0;\by)=\Phi(\tau-T_0,0;\by)\Phi(-T_0,0;\by)^{-1}$.
		
		 Taking each of the columns of $\Phi(-T_0,0;\by)^{-1}$ for $\bphi(0)$, we obtain, first with the matrix norm $\|\cdot \|_1$ and the vector norm $\|\cdot\|_1$, and then also for $\|\cdot\|_2$  with a different constant as all matrix and vector norms are equivalent,
		\begin{eqnarray*}
		\|P_{S_\tau (S_\theta\bx)}	\Phi(\tau,0;S_\theta \bx)\|&=&
		\|P_{S_t\by}\Phi(t,0;\by)	\Phi(-T_0,0;\by)^{-1}\|\\
		&\le&C'e^{-\kappa \tau}.
		\end{eqnarray*}
		Using \eqref{esti} completes the proof.
		\end{proof}

The right-hand side of \eqref{expab} is integrable over $\tau\in [0,\infty)$. Hence, by Lebesgue's dominated convergence theorem, the function $g_T(\theta,\bx)$, see (\ref{g2}), converges pointwise for $T\to\infty$ for $|\theta|\le \theta_0$.
This shows with \eqref{g1} that $M_1(\bx)=\lim_{T\to\infty}g_T(0,\bx)$ is well defined and  symmetric.

Also, using (\ref{g1}), (\ref{I}) and (\ref{t0}), we have  
\begin{eqnarray*}
	\frac{d}{d \theta}g_T(\theta,\bx)
	&=&\Phi(T+\theta,\theta; \bx)^T P_{S_{T+\theta}\bx}^T
	B(S_{T+\theta}\bx)P_{S_{T+\theta}\bx}\Phi(T+\theta,\theta; \bx)
	-P_{S_\theta \bx}^TB(S_\theta \bx)P_{S_\theta \bx}\\
	&&
	-D\bff(S_\theta \bx)^T 
	\int_\theta^{T+\theta} \Phi(\tau,\theta; \bx)^T P_{S_\tau \bx}^T
	B(S_\tau \bx)P_{S_\tau \bx}\Phi(\tau,\theta; \bx)\,d\tau\\
	&&-
	\int_\theta^{T+\theta} \Phi(\tau,\theta; \bx)^T P_{S_\tau \bx}^T
	B(S_\tau \bx)P_{S_\tau \bx}\Phi(\tau,\theta; \bx)\,d\tau  D\bff(S_\theta \bx)\\
	&=&\Phi(T,0; S_\theta \bx)^T P_{S_{T}(S_\theta\bx)}^T
	B(S_{T+\theta}\bx)P_{S_{T}(S_\theta\bx)}\Phi(T,0;S_\theta  \bx)
	-P_{S_\theta \bx}^TB(S_\theta \bx)P_{S_\theta \bx}\\
	&&
	-D\bff(S_\theta \bx)^T 
	\int_0^{T} \Phi(\tau,0; S_\theta \bx)^T P_{S_{\tau}(S_\theta \bx)}^T
	B(S_{\tau+\theta} \bx)P_{S_{\tau}(S_\theta \bx)}\Phi(\tau,0;S_\theta  \bx)\,d\tau\\
	&&-
	\int_0^{T} \Phi(\tau,0; S_\theta \bx)^T P_{S_{\tau}(S_\theta \bx)}^T
	B(S_{\tau+\theta} \bx)P_{S_{\tau}(S_\theta \bx)}\Phi(\tau,0;S_\theta \bx)\,d\tau  D\bff(S_\theta \bx)
	\end{eqnarray*}
	by \eqref{eq0.3}.
The right-hand side
converges uniformly for $|\theta|\le \theta_0$ as $T\to \infty$ by (\ref{expab}).
Hence, we can exchange limit and derivative, obtaining for $|\theta|<\theta_0$, again with \eqref{expab},
\begin{eqnarray}
	\lefteqn{
	\frac{d}{d \theta}\lim_{T\to \infty}g_T(\theta,\bx)}\nonumber\\
	&=&\lim_{T\to \infty}\frac{d}{d \theta}g_T(\theta,\bx)\nonumber\\
	&=&-P_{S_\theta \bx}^TB(S_\theta \bx)P_{S_\theta \bx}\nonumber\\
	&&
	-D\bff(S_\theta \bx)^T 
	\int_0^{\infty} \Phi(\tau,0; S_\theta \bx)^T P_{S_{\tau}(S_\theta \bx)}^T
	B(S_{\tau+\theta} \bx)P_{S_{\tau}(S_\theta \bx)}\Phi(\tau,0;S_\theta  \bx)\,d\tau\nonumber\\
	&&-
	\int_0^{\infty} \Phi(\tau,0;S_\theta \bx)^T P_{S_{\tau}(S_\theta \bx)}^T
	B(S_{\tau+\theta} \bx)P_{S_{\tau}(S_\theta \bx)}\Phi(\tau,0; S_\theta \bx)\,d\tau  D\bff(S_\theta \bx).\label{help33}
\end{eqnarray}

Altogether, we thus have
\begin{eqnarray*}
	M_1'(\bx)&=&\frac{d}{d\theta} M_1(S_\theta \bx)\bigg|_{\theta=0}\\
	&=&\frac{d}{d\theta} \lim_{T\to\infty}\left[
	\int_0^T \Phi(\tau,0;S_{\theta} \bx)^T P_{S_\tau(S_\theta\bx)}^T
	B(S_\tau S_\theta \bx)P_{S_\tau(S_\theta\bx)}\Phi(\tau,0;S_{\theta} \bx)\,d\tau\right]\Bigg|_{\theta=0}\\
	&=&\frac{d}{d\theta} \lim_{T\to\infty}g_T(\theta,\bx)\bigg|_{\theta=0}
	\mbox{ by (\ref{g2})}\\
	&=&
	-P_\bx^T
	B(\bx)P_\bx
	- D\bff( \bx)^T M_1(\bx)
	-
	M_1(\bx)D\bff(\bx)\text{ by (\ref{help33})}\\
	&=&
	-
	C(\bx)
	- D\bff( \bx)^T M_1(\bx)	-
	M_1(\bx)D\bff(\bx)\\
	&&
	-\frac{M_1(\bx)\bff(\bx)\bff(\bx)^T(D\bff(\bx)+D\bff(\bx)^T)+(D\bff(\bx)+D\bff(\bx)^T)\bff(\bx)\bff(\bx)^TM_1(\bx)}{\|\bff(\bx)\|^2}.
\end{eqnarray*}
The last term is zero since $\Phi(t,0;\bx)\bff(\bx)=\bff(S_t\bx)$ and thus
\begin{eqnarray*}
M_1(\bx)\bff(\bx)&=&\int_0^\infty  \Phi(t,0;\bx)^TC(S_t\bx)\Phi(t,0;\bx)\bff(\bx)\,dt\\
&=&\int_0^\infty  \Phi(t,0;\bx)^TP_{S_t\bx}^TB(S_t\bx)P_{S_t\bx}\bff(S_t\bx)\,dt\\
&=&\bnull
\end{eqnarray*}
using $P_{S_t\bx}\bff(S_t\bx)=\bnull$. Similarly we have also $\bff(\bx)^TM_1(\bx)=0$.

This shows the matrix equation $LM_1(\bx)=-C(\bx)$ and thus (\ref{PDE}).

\vspace{0.3cm}

\noindent
\underline{\bf Step 5: Smoothness of $M_1$ }
\vspace{0.2cm}

To prove that $M_1\in C^{\sigma-1}(A(x_0),\mathbb S^n)$, we will define $\bpsi(t,\bx):=P_{S_t\bx}\bphi(t)$, where $\bphi(t)$ is a solution of the first variation equation $\dot{\bphi}(t)=D\bff(S_t\bx)\bphi(t)$. 
 We will show by induction with respect to $|\balpha|$ that
\begin{eqnarray}
\|\partial_\bx^\balpha( \bpsi(t,\bx))\|
&\le&c_\balpha e^{-\kappa_0 t}\max_{\bnull\le \bbeta\le \balpha}\|\partial_\bx^\bbeta (\bpsi(0,\bx))\|\label{state0}
\end{eqnarray}
for all $|\balpha|\le \sigma-1$, $\bx\in U$ and $t\ge 0$, where $\kappa_0:=\frac{\kappa}{2}$.
For $\balpha=\bnull$, \eqref{state0} follows directly from Lemma \ref{help}.

For $\bx\in U$ define $\bphi_0(t,\bx)=\frac{\bff(S_t\bx)}{\|\bff(S_t\bx)\|^2}$ and 
$\ba=\bff(\bx)$. For $i=1,\ldots,n$ let  $\bphi_i(0,\bx)=P_\bx \be_i$
%
and let $\bphi_i(t,\bx)$ be a solution of $\dot{\bphi}(t)=D\bff(S_t\bx)\bphi(t)$.
Then $$\Psi(t,\bx)=\left(a_1\bphi_0(t,\bx)+P_{S_t\bx}\bphi_1(t,\bx),\ldots,a_n\bphi_0(t,\bx)+P_{S_t\bx}\bphi_n(t,\bx)\right)$$
 is the principal fundamental matrix solution of 
\begin{eqnarray}
\frac{d}{dt}\Psi(t,\bx)&=&\tA(t,\bx)\Psi(t,\bx)\label{3.41}
\end{eqnarray}
where $\tA(t,\bx)=D\bff(S_t\bx)-\frac{\bff(S_t\bx)\bff(S_t\bx)^T}{\|\bff(S_t\bx)\|^2}(D\bff(S_t\bx)^T+D\bff(S_t\bx))$. Indeed, 
it can be shown directly that $\bphi_0(t,\bx)=\frac{\bff(S_t\bx)}{\|\bff(S_t\bx)\|^2}$ is a solution of \eqref{3.41} and  in a similar way to \eqref{psi-r} that $P_{S_t\bx}\bphi_i(t,\bx)$ for $i=1,\ldots,n$ are solutions of \eqref{3.41}, see also \cite{leonov2}. Note that 
$$\Psi(0,\bx)=\bphi_0(0,\bx)\ba^T+P_\bx I=\frac{\bff(\bx)\bff(\bx)^T}{\|\bff(\bx)\|^2}+P_\bx I=I.$$

We have with Lemma \ref{help} 
\begin{eqnarray*}
	\|P_{S_t\bx}\Psi(t,\bx)\|_1&=&
	\max_{j=1,\ldots,n} \|a_jP_{S_t\bx}\bphi_0(t,\bx)+P_{S_t\bx}\bphi_j(t,\bx)\|_1\\
	&=&
	\max_{j=1,\ldots,n} \| P_{S_t\bx}\bphi_j(t,\bx)\|_1\\
	&\le&
	Ce^{-2\kappa_0 t}
	\max_{j=1,\ldots,n} \| P_{\bx}\bphi_j(0,\bx)\|_1\\
	&=&
	Ce^{-2\kappa_0 t}
	\max_{j=1,\ldots,n} \|a_jP_{\bx}\bphi_0(0,\bx)+P_{\bx}\bphi_j(0,\bx)\|_1\\
	&=&
	Ce^{-2\kappa_0 t}\|P_{\bx}\Psi(0,\bx)\|_1,
	\end{eqnarray*}
	and, as all norms are equivalent and $\Psi(0,\bx)=I$, with a different constant
	\begin{eqnarray}
	\|P_{S_t\bx}\Psi(t,\bx)\|	&\le&
	Ce^{-2\kappa_0 t}
	\label{Psi-t}
\end{eqnarray}
 for all $t\ge 0$ and all $\bx\in U$. 

We will now show the estimate
\begin{eqnarray}
	\|P_{S_t\bx}\Psi(t,\bx)\Psi(s,\bx)^{-1}\|&\le&Ce^{-2\kappa_0 (t-s)}\label{Psi-ts}
\end{eqnarray} for all $t\ge 0$, $0\le s\le t$ and all $\bx\in U$.
From \eqref{Psi-t} we obtain by considering the point $S_s\bx$ and the time $t-s$
\begin{eqnarray}
	\|P_{S_{t-s}S_s\bx}\Psi(t-s,S_s\bx)\|&\le&Ce^{-2\kappa_0 (t-s)}.\label{helpi1}
	\end{eqnarray}
Denoting the transition matrix from $s$ to $t$ for \eqref{3.41} by $\Psi(t,s;\bx)$, we have
with \eqref{eq0.3} 
\begin{eqnarray}
\Psi(t-s,S_s\bx)&=&\Psi(t-s,0;S_s\bx)\nonumber\\
&=&\Psi(t,s;\bx)\nonumber\\
&=&\Psi(t,\bx)\Psi(s,\bx)^{-1}\label{helpi2}\\
\|P_{S_t\bx}\Psi(t,\bx)\Psi(s,\bx)^{-1}\|&=&\|P_{S_{t-s}S_s\bx}\Psi(t-s, S_s\bx)\| \text{ by \eqref{helpi2} }\nonumber\\
&\le&Ce^{-2\kappa_0 (t-s)}\nonumber
\end{eqnarray}
by   \eqref{helpi1}. 
This shows \eqref{Psi-ts}.

	%

Now we assume \eqref{state0} is true for all $\balpha'$ with $|\balpha'|\le k-1$ and seek to show it for $|\balpha|=k\le \sigma-1$. 
We will write
\begin{eqnarray}
\partial_\bx^\balpha \bpsi(t,\bx)&=&
\left[I-\frac{\bff(S_t\bx)\bff(S_t\bx)^T}{\|\bff(S_t\bx)\|^2}+\frac{\bff(S_t\bx)\bff(S_t\bx)^T}{\|\bff(S_t\bx)\|^2}\right]
\partial_\bx^\balpha \bpsi(t,\bx)\nonumber\\
&=&P_{S_t\bx}\partial_\bx^\balpha \bpsi(t,\bx)
+\frac{\bff(S_t\bx)\bff(S_t\bx)^T}{\|\bff(S_t\bx)\|^2}
\partial_\bx^\balpha \bpsi(t,\bx)\label{first-second}
\end{eqnarray}
and show that each term satisfies the exponential bound in \eqref{state0}.

For the second term of \eqref{first-second}, we have
\begin{eqnarray*}
\frac{\bff(S_t\bx)\bff(S_t\bx)^T}{\|\bff(S_t\bx)\|^2}\bpsi(t,\bx)&=&
\frac{\bff(S_t\bx)\bff(S_t\bx)^T}{\|\bff(S_t\bx)\|^2}
P_{S_t\bx}\bphi(t)\ =\ \bnull
\end{eqnarray*}
for all $t\ge 0$ and $\bx\in U$
since $\bff(S_t\bx)^TP_{S_t\bx}=\bnull$.
Hence, 
\begin{eqnarray*}\bnull&=&
\partial_\bx^\balpha\left(
\frac{\bff(S_t\bx)\bff(S_t\bx)^T}{\|\bff(S_t\bx)\|^2}\bpsi(t,\bx)\right)\\
&=&\frac{\bff(S_t\bx)\bff(S_t\bx)^T}{\|\bff(S_t\bx)\|^2}\partial_\bx^\balpha\bpsi(t,\bx)
\\
&&+\sum_{\balpha_1+\balpha_2=\balpha,|\balpha_1|\ge 1} c_{\balpha_1} \partial^{\balpha_1}_\bx \left(
\frac{\bff(S_t\bx)\bff(S_t\bx)^T}{\|\bff(S_t\bx)\|^2}\right)	\partial^{\balpha_2}_\bx \bpsi(t,\bx)\\
\frac{\bff(S_t\bx)\bff(S_t\bx)^T}{\|\bff(S_t\bx)\|^2}\partial_\bx^\balpha\bpsi(t,\bx)
&=&-\sum_{\balpha_1+\balpha_2=\balpha,|\balpha_1|\ge 1} c_{\balpha_1} \partial^{\balpha_1}_\bx \left(
\frac{\bff(S_t\bx)\bff(S_t\bx)^T}{\|\bff(S_t\bx)\|^2}\right)	\partial^{\balpha_2}_\bx \bpsi(t,\bx).
\end{eqnarray*}
By induction assumption and smoothness of $\bff$ in the compact, positively invariant set $U$, the norm of the right-hand side is smaller than
$c e^{-\kappa_0 t}\max_{\bnull\le \bbeta<\balpha}\|\partial_\bx^\bbeta\bpsi(0,\bx)\|$, and thus so is the left-hand side. This shows the exponential bound on the norm of the second term of \eqref{first-second}.

For the first term of \eqref{first-second}, we have
\begin{eqnarray}
	\frac{d}{dt}
	\partial_\bx^{\balpha}\bpsi(t,\bx)
	&=&\partial_\bx^{\balpha}  \frac{d}{dt}\bpsi(t,\bx)\nonumber\\
	&=&\partial_\bx^{\balpha}\left[\tA(t,\bx)\bpsi(t,\bx)\right]\nonumber\\
	&=&\tA(t,\bx)\partial_\bx^{\balpha}\bpsi(t,\bx)\nonumber\\
	&&
	+\underbrace{\sum_{\balpha_1+\balpha_2=\balpha,|\balpha_1|\ge 1} c_{\balpha_1} \partial^{\balpha_1}_\bx \tA(t,\bx)
	\partial^{\balpha_2}_\bx \bpsi(t,\bx)}_{=:g(t,\bx)}.\label{eq11}
\end{eqnarray} Note that
we could exchange $ \partial_\bx^{\balpha}$ and $\frac{d}{dt} $ above
since $P_{S_t\bx}$ and  $\bphi$ are smooth enough,
cf. e.g. \cite{hartman}, Chapter V, Theorem 4.1.

From the induction assumption we know that for all $|\balpha_2|\le k-1$
$$
\|\partial_\bx^{\balpha_2}\bpsi(t,\bx)\|
\le c_{\balpha_2} e^{-\kappa_0 t}\max_{\bnull\le \bbeta\le\balpha_2}\|\partial_\bx^\bbeta \bpsi(0,\bx)\| .$$
From the definition of $\tA$, since $\bff\in C^\sigma(\mathbb R^n,\mathbb R^n)$ and $U$ is compact and positively invariant, there is a constant bounding  $\|\partial^{\balpha_1}_\bx \tA(t,\bx)\|$ for all $
|\balpha_1|\le \sigma-1$, $t\ge 0$ and $\bx\in U$.
This shows altogether that \begin{eqnarray}
	\|g(t,\bx)\|&\le&Ce^{-\kappa_0 t}\max_{\bnull\le \bbeta < \balpha}\|\partial_\bx^\bbeta \bpsi(0,\bx)\| \label{g-est}
	\end{eqnarray}
	for all $t\ge 0 $ and $\bx\in U$.
	
Using the variation of the constant formula, the solution  $\bz(t,\bx)=\partial_\bx^{\balpha}\bpsi(t,\bx)$ of
\begin{eqnarray}
\frac{d}{dt}\bz(t,\bx)&=&
\tA(t,\bx)\bz(t,\bx)+g(t,\bx),\label{eq-g}
\end{eqnarray} 
see \eqref{eq11}, satisfies
\begin{eqnarray*}
	\bz(t,\bx)&=&
	\Psi(t,\bx)\bz(0,\bx)+\int_0^t\Psi(t,\bx)\Psi(s,\bx)^{-1}g(s,\bx)\,ds.
\end{eqnarray*}
Application of the projection $P_{S_t\bx}$ from the left on both sides gives
\begin{eqnarray*}
	P_{S_t\bx}\partial_\bx^{\balpha}\bpsi(t,\bx)&=&
	P_{S_t\bx}\Psi(t,\bx)\partial_\bx^{\balpha}\bpsi(0,\bx)+\int_0^tP_{S_t\bx}\Psi(t,\bx)\Psi(s,\bx)^{-1}g(s,\bx)\,ds.
\end{eqnarray*}
Then we have with \eqref{Psi-t}, \eqref{Psi-ts} and \eqref{g-est} 
\begin{eqnarray*}
\|		P_{S_t\bx}\partial_\bx^{\balpha}\bpsi(t,\bx)\|
&\le&\left(Ce^{-2\kappa_0 t}+\int_0^tCe^{-2\kappa_0 (t-s)} e^{-\kappa_0 s}\,ds\right)
\max_{\bnull\le \bbeta\le\balpha}\|\partial_\bx^\bbeta \bpsi(0,\bx)\| \\
&=&\left(Ce^{-2\kappa_0 t}+Ce^{-2\kappa_0 t} \int_0^t e^{\kappa_0 s}\,ds\right)
\max_{\bnull\le \bbeta\le\balpha}\|\partial_\bx^\bbeta \bpsi(0,\bx)\|\\
&\le&\left(Ce^{-2\kappa_0 t}+\frac{C}{\kappa_0}e^{-2\kappa_0 t}  e^{\kappa_0 t}\right)
\max_{\bnull\le \bbeta\le\balpha}\|\partial_\bx^\bbeta \bpsi(0,\bx)\|\\	
&\le& c_\balpha e^{-\kappa_0 t}
\max_{\bnull\le \bbeta\le\balpha}\|\partial_\bx^\bbeta \bpsi(0,\bx)\|.
\end{eqnarray*}
This shows the bound on the first term of \eqref{first-second} and thus \eqref{state0}.

Next, we show that $\int_0^T \partial_\bx^\balpha (\Phi(t,0;\bx)^T P_{S_t\bx}^TB(S_t\bx)P_{S_t\bx}\Phi(t,0;\bx))\,dt$
converges uniformly with respect to $\bx$ as $T\to \infty$
for $1\le |\balpha|\le\sigma-1$. Let
$\bx\in A(\Omega)$ and let $O$ be a bounded, open neighborhood of $\bx$,
such that $\overline{O}\subset A(\Omega)$. Since $\overline{O}$ is
compact,  there is a
$T_0\in \mathbb R^+_0$ such that $S_{T_0+t}\overline{O}\subset
U$ holds for all $t\ge 0$. Hence, it is sufficient to show the statement for all $\bx\in U$. 

We can write the $i$-th column of $P_{S_t\bx}\Phi(t,0,\bx)$ as
$\bpsi(t,\bx)=P_{S_t\bx}\Phi(t,0;\bx)\be_i=P_{S_t\bx}\bphi(t,\bx)$ with $\bphi(0,\bx)=\be_i$. Thus, $\partial_{\bx}^\balpha \bpsi(0,\bx)=\partial_\bx^\balpha P_\bx \be_i$, which can be bounded by a constant for all $\bx\in U$ and $|\balpha|\le \sigma-1$ by the smoothness of $\bff$.
Similarly, $\partial_\bx^\balpha B(S_t\bx)$ can be bounded by a constant for all $\bx\in U$, $t\ge 0$ and $|\balpha|\le \sigma-1$. Altogether, we have by    (\ref{state0})  \begin{eqnarray*}\int_0^T \|\partial_\bx^\balpha (\Phi(t,0;\bx)^T P_{S_t\bx}^TB(S_t\bx)P_{S_t\bx}\Phi(t,0;\bx))\|\,dt&\le&
	\int_0^T \tilde{c} e^{-2\kappa_0 t}\,dt
\end{eqnarray*}
for all $\bx\in U$ and all $T\ge 0$. Hence, $\int_0^T \partial_\bx^\balpha (\Phi(t,0;\bx)^T P_{S_t\bx}^TB(S_t\bx)P_{S_t\bx}\Phi(t,0;\bx))\,dt $ converges uniformly as $T\to \infty$.
This proves that $M_1\in C^{\sigma-1}(A(\Omega),\mathbb S^n)$.

\vspace{0.3cm}

\noindent
\underline{\bf Step 6: positive definiteness}
\vspace{0.2cm}

To show the positive definiteness of $M$, fix $\bx\in A(\Omega)$ and consider a general  
$$\mathbb R^n\ni\bw=\left(I-\frac{\bff(\bx)\bff(\bx)^T}{\|\bff(\bx)\|^2}
+\frac{\bff(\bx)\bff(\bx)^T}{\|\bff(\bx)\|^2}\right)\bw
	=\bv+c\frac{\bff(\bx)}{\|\bff(\bx)\|^2}$$ with $\bv=P_\bx \bw$, so $\bv\perp\bff(\bx)$, and $c=\bff(\bx)^T\bw$. Hence, using $C(\bx)=P_\bx^TB(\bx)P_\bx$, we have
	\begin{eqnarray*}\bw^TM(\bx)\bw&=&
\int_0^\infty [P_{S_t\bx}\Phi(t,0;\bx)\bw]^TB(S_t\bx)[P_{S_t\bx}\Phi(t,0;\bx)\bw]\,dt	+c_0
\bw^T\bff(\bx)\bff(\bx)^T\bw\\
&=&
\int_0^\infty [P_{S_t\bx}\Phi(t,0;\bx)\bw]^TB(S_t\bx)[P_{S_t\bx}\Phi(t,0;\bx)\bw]\,dt	+c_0c^2\\
&\ge &0
	\end{eqnarray*}
	due to the positive definiteness of $B$ and $c_0>0$.
	We seek to show that the term is only $0$ if $\bw=\bnull$.
	
	Let us assume that the term is zero, i.e. both summands are zero. The first term, since $B$ is positive definite, is only zero if
	$P_{S_t\bx}\Phi(t,0;\bx)\bw=\bnull$ for all $t\ge 0$. In particular, for $t=0$ we have $\bnull=P_\bx \bw=\bv$. If the second term is zero, then, since $c_0>0$, $c=0$, which together yields $\bw=\bnull$.
	
	This proves the theorem.
\end{proof}

\section{Uniqueness}
\label{uni}

To show uniqueness of solutions to \eqref{PDE} and \eqref{initial} in Theorem \ref{th:uni}, let us first state the following lemma.

\begin{lemma}\label{lem}
Denote by $\bphi_1$ and $\bphi_2$ two solutions of $\dot{\bphi}(t)=D\bff(S_t\bx)\bphi(t)$.
Let $M\in C^1(\mathbb R^n,\S^n)$ such that $M(\bx)$ is positive definite for all $\bx\in \mathbb R^n$. 

Then we have for any $\bx\in \mathbb R^n$ with $\bff(\bx)\not=\bnull$ and all $t\ge 0$, 
denoting $\bphi_0(t)=\frac{\bff(S_t\bx)}{\|\bff(S_t\bx)\|^2}$,
\begin{eqnarray}
\frac{d}{dt}\left[\bphi_1(t)^TP_{S_t\bx}^TM(S_t\bx)P_{S_t\bx}\bphi_2(t)\right]
&=&\bphi_1(t)^TP_{S_t\bx}^TLM(S_t\bx)P_{S_t\bx}\bphi_2(t),\label{4.1}\\
\frac{d}{dt}\left[\bphi_0(t)^T M(S_t\bx) \bphi_0(t)\right]
&=&\bphi_0(t)^T LM(S_t\bx) \bphi_0(t),\label{4.2}\\
\frac{d}{dt}\left[\bphi_1(t)^TP_{S_t\bx}^TM(S_t\bx) \bphi_0(t)\right]
&=&\bphi_1(t)^TP_{S_t\bx}^TLM(S_t\bx)\bphi_0(t),\label{4.3} \\
\frac{d}{dt}\left[\bphi_0(t)^TM(S_t\bx)P_{S_t\bx} \bphi_1(t)\right]
&=&\bphi_0(t)^TLM(S_t\bx)P_{S_t\bx}\bphi_1(t).\label{4.4}
\end{eqnarray}
\end{lemma}
\begin{proof}Note that $\bff(S_t\bx)\not=\bnull$ for all $t\ge 0$ since $\bff(\bx)\not=\bnull$.

	For the first statement we calculate
\begin{eqnarray*}\lefteqn{
\frac{d}{dt}\left[\bphi_1(t)^TP_{S_t\bx}^TM(S_t\bx)P_{S_t\bx}\bphi_2(t)\right]}\\
&=&\bphi_1(t)^TD\bff(S_t\bx)^T\left(I-\frac{\bff(S_t\bx)\bff(S_t\bx)^T}{\|\bff(S_t\bx)\|^2}\right)M(S_t\bx)P_{S_t\bx}\bphi_2(t)\\
&&-	\bphi_1(t)^T\frac{D\bff(S_t\bx)\bff(S_t\bx)\bff(S_t\bx)^T+\bff(S_t\bx)\bff(S_t\bx)^TD\bff(S_t \bx)^T}{\|\bff(S_t\bx)\|^2}M(S_t\bx)P_{S_t\bx}\bphi_2(t)\\
	&&+\bphi_1(t)^T\frac{\bff(S_t\bx)\bff(S_t\bx)^T(D\bff(S_t\bx)+D\bff(S_t\bx)^T)\bff(S_t\bx)\bff(S_t\bx)^T}{\|\bff(S_t\bx)\|^4}M(S_t\bx)P_{S_t\bx}\bphi_2(t)\\
&&+\bphi_1(t)^TP_{S_t\bx}^TM'(S_t\bx)P_{S_t\bx}\bphi_2(t)\\	
&&-	\bphi_1(t)^TP_{S_t\bx}^TM(S_t\bx)\frac{D\bff(S_t \bx)\bff(S_t\bx)\bff(S_t\bx)^T+\bff(S_t\bx)\bff(S_t\bx)^TD\bff(S_t\bx)^T}{\|\bff(S_t\bx)\|^2}\bphi_2(t)\\
&&+\bphi_1(t)^TP_{S_t\bx}^TM(S_t\bx)
\frac{\bff(S_t\bx)\bff(S_t\bx)^T(D\bff(S_t\bx)+D\bff(S_t\bx)^T)\bff(S_t\bx)\bff(S_t\bx)^T}{\|\bff(S_t\bx)\|^4}\bphi_2(t)\\
&&+\bphi_1(t)^TP_{S_t\bx}^TM(S_t\bx)\left(I-\frac{\bff(S_t\bx)\bff(S_t\bx)^T}{\|\bff(S_t\bx)\|^2}\right)D\bff(S_t\bx)\bphi_2(t)\\
&=&\bphi_1(t)^T\left(I-\frac{\bff(S_t\bx)\bff(S_t\bx)^T}{\|\bff(S_t\bx)\|^2}\right)
\bigg[M'(S_t\bx)+D\bff(S_t\bx)^TM(S_t\bx)+M(S_t\bx)D\bff(S_t\bx)\\
&&
-\frac{M(S_t\bx)\bff(S_t\bx)\bff(S_t\bx)^T(D\bff(S_t\bx)+D\bff(S_t\bx)^T)}{\|\bff(S_t\bx)\|^2}
\\
&&
-\frac{(D\bff(S_t\bx)+D\bff(S_t\bx)^T)\bff(S_t\bx)\bff(S_t\bx)^T M(S_t\bx)}{\|\bff(S_t\bx)\|^2}
\bigg]
\left(I-\frac{\bff(S_t\bx)\bff(S_t\bx)^T}{\|\bff(S_t\bx)\|^2}\right)\bphi_2(t)
\end{eqnarray*}
which can be verified in a straight-forward calculation.

For the second statement we calculate
\begin{eqnarray*}\lefteqn{\frac{d}{dt}\left[\frac{\bff(S_t\bx)^T}{\|\bff(S_t\bx)\|^2} M(S_t\bx) \frac{\bff(S_t\bx)}{\|\bff(S_t\bx)\|^2}\right]}\\
&=&\frac{\bff(S_t\bx)^TD\bff(S_t\bx)^T}{\|\bff(S_t\bx)\|^2} M(S_t\bx) \frac{\bff(S_t\bx)}{\|\bff(S_t\bx)\|^2}\\
&&-\frac{\bff(S_t\bx)^T(D\bff(S_t\bx)+D\bff(S_t\bx)^T)\bff(S_t\bx)\bff(S_t\bx)^T}{\|\bff(S_t\bx)\|^4} M(S_t\bx) \frac{\bff(S_t\bx)}{\|\bff(S_t\bx)\|^2}\\
&&+\frac{\bff(S_t\bx)^T}{\|\bff(S_t\bx)\|^2} M'(S_t\bx) \frac{\bff(S_t\bx)}{\|\bff(S_t\bx)\|^2}\\
&&+\frac{\bff(S_t\bx)^T}{\|\bff(S_t\bx)\|^2} M(S_t\bx) \frac{D\bff(S_t\bx)\bff(S_t\bx)}{\|\bff(S_t\bx)\|^2}\\
&&- \frac{\bff(S_t\bx)^T}{\|\bff(S_t\bx)\|^2}M(S_t\bx)
\frac{\bff(S_t\bx)\bff(S_t\bx)^T(D\bff(S_t\bx)+D\bff(S_t\bx)^T)\bff(S_t\bx)}{\|\bff(S_t\bx)\|^4} \\
&=&\frac{\bff(S_t\bx)^T}{\|\bff(S_t\bx)\|^2}
\bigg[M'(S_t\bx)+D\bff(S_t\bx)^TM(S_t\bx)+M(S_t\bx)D\bff(S_t\bx)\\
&&
-\frac{M(S_t\bx)\bff(S_t\bx)\bff(S_t\bx)^T(D\bff(S_t\bx)+D\bff(S_t\bx)^T)}	{\|\bff(S_t\bx)\|^2}\\
&&
-\frac{(D\bff(S_t\bx)+D\bff(S_t\bx)^T)\bff(S_t\bx)\bff(S_t\bx)^T M(S_t\bx)}{\|\bff(S_t\bx)\|^2}
\bigg]\frac{\bff(S_t\bx)^T}{\|\bff(S_t\bx)\|^2}
\end{eqnarray*}
which can be verified in a straight-forward calculation. The last statements can be proven in a similar way as the previous two.
\end{proof}

\begin{theorem}	\label{th:uni}	The solution $M$ of \eqref{PDE} and \eqref{initial}, see Theorem \ref{main}, is unique in $A(\Omega)$. 
	\end{theorem}
	\begin{proof}
	Let $M_1$ and $M_2$ be two solutions of \eqref{PDE} and \eqref{initial}. Let $\bx\in A(\Omega)	$ and let $\Phi(t,0;\bx)$ be the principal fundamental matrix solution of $\dot{\bphi}(t)=D\bff(S_t\bx)\bphi(t)$ with $\Phi(0,0;\bx)=I$. 
	We want to show that $\bu_1^T[M_1(\bx)-M_2(\bx)]\bu_2=0$ for all $\bu_1,\bu_2\in \mathbb R^n$. We write $$\bu_i=P_\bx \bu_i+c_i\frac{\bff(\bx)}{\|\bff(\bx)\|^2}$$
	with  $c_i=\bff(\bx)^T\bu_i$.
	 Then we have, denoting $\bphi_0(t)=\frac{\bff(S_t\bx)}{\|\bff(S_t\bx)\|^2}$,
	 \begin{eqnarray}
	 	\lefteqn{
\bu_1^T[M_1(\bx)-M_2(\bx)]\bu_2}\nonumber\\&=&
\bu_1^TP_\bx^T[M_1(\bx)-M_2(\bx)]P_\bx\bu_2+
c_1\bphi_0(0)^T[M_1(\bx)-M_2(\bx)]P_\bx\bu_2\nonumber\\
&&+
c_2\bu_1^TP_\bx^T[M_1(\bx)-M_2(\bx)]\bphi_0(0)
+c_1c_2 \bphi_0(0)^T[M_1(\bx)-M_2(\bx)]\bphi_0(0).\label{4terms}
\end{eqnarray}
We will show that each term in \eqref{4terms} is zero.

From \eqref{4.1} of Lemma \ref{lem} we have for $i=1,2$
\begin{eqnarray*}
\frac{d}{dt}\left[\Phi(t,0;\bx)^TP_{S_t\bx}^TM_i(S_t\bx)P_{S_t\bx}\Phi(t,0;\bx)\right]
&=&\Phi(t,0;\bx)^TP_{S_t\bx}^TLM_i(S_t\bx)P_{S_t\bx}\Phi(t,0;\bx)\\
&=&-\Phi(t,0;\bx)^TP_{S_t\bx}^TC(S_t\bx)P_{S_t\bx}\Phi(t,0;\bx).
	\end{eqnarray*}
	Hence, by subtracting the equations for $M_1-M_2$ we obtain
	\begin{eqnarray*}
		\frac{d}{dt}\left[\Phi(t,0;\bx)^TP_{S_t\bx}^T[M_1(S_t\bx)-M_2(S_t\bx)]P_{S_t\bx}\Phi(t,0;\bx)\right]&=&0
		\end{eqnarray*}
		and by integrating
		\begin{eqnarray*}
\|	P_{\bx}^T[M_1(\bx)-M_2(\bx)]P_{\bx}\|&=&		\left\|\Phi(t,0;\bx)^TP_{S_t\bx}^T[M_1(S_t\bx)-M_2(S_t\bx)]P_{S_t\bx}\Phi(t,0;\bx)\right\|\\
	&\le&	\|P_{S_t\bx}\Phi(t,0;\bx)\|^2\|M_1(S_t\bx)-M_2(S_t\bx)\|\\
	&\to&0
	\end{eqnarray*}
	as $t\to \infty$ 
	since $M_i$ are continuous, $\overline{\bigcup_{t\ge 0}S_t\bx}$ is compact and  $\|P_{S_t\bx}\Phi(t,0;\bx)\|$ is exponentially decreasing to zero by Lemma \ref{help} (note that there exists $T_0\ge 0$ such that for all $t\ge T_0$ we have $S_t\bx \in U$ as $\bx\in A(\Omega)$).
		
	Similarly, using \eqref{4.3} and \eqref{4.4}, we have $P_{\bx}^T[M_1(\bx)-M_2(\bx)]\bphi_0(0)=\bnull$ as well as $	\bphi_0(0)^T[M_1(\bx)-M_2(\bx)]P_{\bx}=\bnull^T$.
	 This shows that the first three terms of
	\eqref{4terms} are zero.
	
	For the last term of \eqref{4terms} we have for $M_i$ satisfying $LM_i(\bx)=-C(\bx)$
	by \eqref{4.2} of Lemma \ref{lem}
	\begin{eqnarray*}
	\frac{d}{dt}\left[\bphi_0(t)^T M_i(S_t\bx) \bphi_0(t)\right]
&=&\bphi_0(t)^T LM_i(S_t\bx) \bphi_0(t)\\
&=&-\bphi_0(t)^TC(S_t\bx) \bphi_0(t)\\
&=&-\bphi_0(t)^TP_{S_t\bx}^TB(S_t\bx)P_{S_t\bx} \bphi_0(t)\\
&=&0
\end{eqnarray*}
since $P_{S_t\bx} \bphi_0(t)=P_{S_t\bx}\frac{ \bff(S_t\bx)}{\|\bff(S_t\bx)\|^2}=\bnull$.
Hence, $\bphi_0(t)^T M_i(S_t\bx) \bphi_0(t)$ is constant along trajectories.
Restricting ourselves to the periodic orbit, this means in particular that there are constants $C_1, C_2$ such that $\frac{\bff(\bp)^T}{\|\bff(\bp)\|^2}M_i(\bp)\frac{\bff(\bp)}{\|\bff(\bp)\|^2}=C_i$ for all $\bp\in \Omega$, $i=1,2$.

Since $\dist(S_t\bx,\Omega)\to 0$ as $t\to \infty$, we have for a general $\bx\in A(\Omega)$
\begin{eqnarray*}
	\frac{\bff(\bx)^T}{\|\bff(\bx)\|^2}M_i(\bx)\frac{\bff(\bx)}{\|\bff(\bx)\|^2}&=&	\frac{\bff(S_t\bx)^T}{\|\bff(S_t\bx)\|^2}M_i(S_t\bx)\frac{\bff(S_t\bx)}{\|\bff(S_t\bx)\|^2}\ \longrightarrow \ C_i
	\end{eqnarray*}
	as $t\to \infty$.	By \eqref{initial}, for $\bx=\bx_0$ we have for $i=1,2$
	\begin{eqnarray*}
C_i&=&			\frac{\bff(\bx_0)^T}{\|\bff(\bx_0)\|^2}M_i(\bx_0)\frac{\bff(\bx_0)}{\|\bff(\bx_0)\|^2}
\ =\ c_0
\end{eqnarray*}
and thus $C_1=C_2$. This means that
$$\bphi_0(\bx)^T[M_1(\bx)-M_2(\bx)]\bphi_0(\bx)=C_1-C_2=0$$
for all $\bx\in A(\Omega)$ and shows that the last term in \eqref{4terms} is zero.
	\end{proof}
	
\section{Conclusions}
\label{con}

In this paper we have presented a matrix-valued PDE with a given value at one point; we have shown existence and uniqueness of a solution, we have established that the solution is of a specific form and that it is a positive definite matrix at each point.

In particular, this shows that the solution is a contraction metric, which implies the existence, uniqueness and exponential stability of a periodic orbit, and determines its basin of attraction. We have thus shown a converse theorem on the existence of a contraction metric for periodic orbits.

By characterizing the contraction metric as solution of a PDE, numerical methods can now be employed for its explicit construction. For example, mesh-free collocation can be used to solve this linear matrix-valued PDE \cite{giesl-wendland}, and error estimates are available. Since even an approximation to the solution of the PDE is a contraction metric, this allows for the explicit construction of a contraction metric.

\begin{appendix}

\section{Gronwall}
\label{Gronwall}

We cite the following lemma from  \cite[Lemma D.2]{sell-you}.
\begin{lemma}\label{gronwall}
 Let $r,K,a\in L_{loc}^1([0,\infty),\mathbb R)$ be nonnegative functions  and let $b\in L^\infty_{loc}([0,\infty),\mathbb R)$ be a continuous nonnegative function such that
$$r(\theta)\le a(\theta)+K(\theta)\int_0^\theta b(s)r(s)\,ds $$
 holds for almost all $\theta\ge 0$.

Then
$$r(\theta)\le a(\theta)+K(\theta)\int_0^\theta a(s) b(s)\,ds\cdot \exp\left(\int_0^\theta K(s)b(s)\,ds\right)$$
holds for almost all $\theta\ge 0$.
\end{lemma}

\section{Proof of Theorem \ref{1.1}}
\label{proof}

In this section we give a sketch of the proof of Theorem \ref{1.1}. It is very similar to the proof of \cite[Theorem 5]{giesl04}, which considers adjacent solutions in direction $\bv$ with $\bv^TM(\bx)\bff(\bx)=0$ while we consider $\bv$ with  $\bv^T\bff(\bx)=0$. Note that a similar result as in Theorem \ref{1.1} with $M(\bx)=I$, so the Euclidean metric, has been proven in \cite{leonov1}.

We now use the notations as in  \cite[Theorem 5]{giesl04} and only highlight the necessary changes; all references are with respect to the proof of  \cite[Theorem 5]{giesl04}.
In \cite[Proposition 7]{giesl04}, which defines the time synchronization $T_p^{p+\eta}$ of the solutions $S_\theta p $ and $S_{T_p^{p+\eta}(\theta)}(p+\eta)$, we replace the first equation by
$$\left(S_{T_p^{p+\eta}(\theta)}(p+\eta)-S_\theta p \right)^Tf(S_\theta p)=0.$$

In the proof we replace
(12) by 
		\begin{eqnarray}
			\|Df(p)-Df(p+\xi)\|\le C_1:=\frac{\lambda_m}{\lambda_M \left(1+2\frac{f_M^2}{f_m^2}\right)}
			\frac{k\nu}{2}\label{B.1}
			\end{eqnarray}
and (14) by
\begin{eqnarray}
\delta'&:=&\min\left(\delta_1,\frac{\sqrt{\lambda_m} }{f_M f_D}\frac{f_m^2}{2},
\frac{\lambda_m^{3/2}f_m^2}{4\lambda_Mf_Mf_D^2\left(1+\frac{f_M^2}{f_m^2}\right)} \frac{k\nu}{2}
 \right);\label{delta}
\end{eqnarray}
 see \cite{giesl04} for the definition of the constants.
(16) is replaced by
$$Q(T,\theta,\eta)=(S_T(p+\eta)-S_{\theta} p)^Tf(S_\theta p)=0.$$

We define
$$A(\theta):=\sqrt{(S_{T(\theta)}(p+\eta)-S_\theta p)^T M(S_\theta p)(S_{T(\theta)}(p+\eta)-S_\theta p)}$$
and define $v(\theta)$ by
$$A(\theta)v(\theta )=S_{T(\theta)}(p+\eta)-S_\theta p.$$
We replace (18) and (19) by
\begin{eqnarray*}
	\partial_\theta Q(T,\theta,\eta)&=&-\|f(S_\theta p)\|^2+A(\theta)v(\theta)^TDf(S_\theta p)f(S_\theta p)\\
	\partial_T Q(T,\theta,\eta)&=&f(S_T(p+\eta))f(S_\theta p)\\
	&=&\|f(S_\theta p)\|^2+A(\theta)\left(\int_0^1 Df(S_\theta p+\lambda A(\theta)v(\theta))\,d\lambda v(\theta)\right)^T f(S_\theta p)\\
	\dot{T}(\theta)&=&1-A(\theta)\frac{ \left(\left(\int_0^1 Df(S_\theta p+\lambda A(\theta)v(\theta))\,d\lambda v(\theta)\right )^T+v(\theta)^TDf(S_\theta p)\right)f(S_\theta p)
	}{\|f(S_\theta p)\|^2+A(\theta)\left(\int_0^1 Df(S_\theta p+\lambda A(\theta)v(\theta))\,d\lambda v(\theta)\right)^T f(S_\theta p)}
\end{eqnarray*}

The first equation in part III becomes thus
\begin{eqnarray*}
	\frac{d}{d\theta}A^2(\theta)&=&(S_{T(\theta)}(p+\eta)-S_\theta p)^TM'(S_\theta p)(S_{T(\theta)}(p+\eta)-S_\theta p)\\
&&+	2(S_{T(\theta)}(p+\eta)-S_\theta p)^TM(S_\theta p)\left(f(S_{T(\theta)}(p+\eta))\dot{T}(\theta)-f(S_\theta p)\right)\\
	&=&(S_{T(\theta)}(p+\eta)-S_\theta p)^TM'(S_\theta p)(S_{T(\theta)}(p+\eta)-S_\theta p)\\
	&&+	2(S_{T(\theta)}(p+\eta)-S_\theta p)^TM(S_\theta p)\cdot\bigg([f(S_{T(\theta)}(p+\eta))-f(S_\theta p)]\\
	&&+(\dot{T}(\theta)-1)f(S_\theta p)
	+(\dot{T}(\theta)-1)[f(S_{T(\theta)}(p+\eta))-f(S_\theta p)]\bigg)\\
	&=&
	A^2(\theta)v(\theta)^TM'(S_\theta p)v(\theta)\\	
	&&+2A^2(\theta) v(\theta)^TM(S_\theta p)\int_0^1 Df(S_\theta p+\lambda A(\theta )v(\theta))\,d\lambda v(\theta)\\
		&&-2A^2(\theta) v(\theta)^TM(S_\theta p)f(S_\theta p)\cdot\\
&&		\frac{ \left(\left(\int_0^1 Df(S_\theta p+\lambda A(\theta)v(\theta))\,d\lambda v(\theta)\right )^T+v(\theta)^TDf(S_\theta p)\right)f(S_\theta p)
		}{\|f(S_\theta p)\|^2+A(\theta)\left(\int_0^1 Df(S_\theta p+\lambda A(\theta)v(\theta))\,d\lambda v(\theta)\right)^T f(S_\theta p)}\\
		&&-2A^3(\theta) v(\theta)^TM(S_\theta p)\int_0^1 Df(S_\theta p+\lambda A(\theta )v(\theta))\,d\lambda v(\theta)\\
		&&\cdot
\frac{ \left(\left(\int_0^1 Df(S_\theta p+\lambda A(\theta)v(\theta))\,d\lambda v(\theta)\right )^T+v(\theta)^TDf(S_\theta p)\right)f(S_\theta p)
}{\|f(S_\theta p)\|^2+A(\theta)\left(\int_0^1 Df(S_\theta p+\lambda A(\theta)v(\theta))\,d\lambda v(\theta)\right)^T f(S_\theta p)}.
\end{eqnarray*}
To obtain a bound on the denominator of the last terms we use  $A(\theta)\frac{f_M f_D}{\sqrt{\lambda_m}}\le \frac{1}{2}f_m^2$ by \eqref{delta}. Also, using $\int_0^1 Df(S_\theta p+\lambda A(\theta)v(\theta))\,d\lambda=Df(S_\theta p )+\int_0^1 [Df(S_\theta p+\lambda A(\theta)v(\theta))-Df(S_\theta p)]\,d\lambda$, we obtain
with \eqref{B.1} and $\|v(\theta)\|\le \frac{1}{\sqrt{\lambda_m}}$
\begin{eqnarray*}
	\frac{d}{d\theta}A^2(\theta)&\le&A^2(\theta)v(\theta)^TM'(S_\theta p)v(\theta)\\
	&&			+2	A^2(\theta)v(\theta)^TM(S_\theta p)Df(S_\theta p)v(\theta)
	+2A^2(\theta)C_1\frac{\lambda_M}{\lambda_m}\\
	&&		-2	A^2(\theta)v(\theta)^TM(S_\theta p)f(S_\theta p)\cdot\\
	&&
					\frac{ v(\theta )^T[Df(S_\theta p)^T+Df(S_\theta p)] f(S_\theta p)		}{\|f(S_\theta p)\|^2+A(\theta)\left(\int_0^1 Df(S_\theta p+\lambda A(\theta)v(\theta))\,d\lambda v(\theta)\right)^T f(S_\theta p)}\\
&&+4A^2(\theta)\frac{\lambda_Mf_M^2C_1}{f_m^2\lambda_m}
+8A^3(\theta)\frac{\lambda_Mf_D^2f_M}{f_m^2\lambda_m^{3/2}}.
\end{eqnarray*}

%
%

Note that $\frac{1}{b+c}=\frac{1}{b}-\frac{c }{b(b+c)}$ holds for all $b,b+c>0$. Using this with
$b=\|f(S_\theta p)\|^2$ and $c=A(\theta)\left(\int_0^1 Df(S_\theta p+\lambda A(\theta)v(\theta))\,d\lambda v(\theta)\right)^T f(S_\theta p)$,
 we have
\begin{eqnarray*}
	\frac{d}{d\theta}A^2(\theta)&\le&2A^2(\theta)L_M(S_\theta p)\\
	&&
	+2A^2(\theta)C_1\frac{\lambda_M}{\lambda_m}+8A^3(\theta)\frac{\lambda_Mf_D^2f_M^3}{f_m^4\lambda_m^{3/2}}\\
	&&+4A^2(\theta)\frac{\lambda_Mf_M^2C_1}{f_m^2\lambda_m}
+8A^3(\theta)\frac{\lambda_Mf_D^2f_M}{f_m^2\lambda_m^{3/2}}\\
%
%
	&\le&-2A^2(\theta)\nu
	+2A^2(\theta)C_1\frac{\lambda_M}{\lambda_m}\left(1+2\frac{f_M^2}{f_m^2}\right)\\
	&&
	+8A^3(\theta )\frac{\lambda_M}{\lambda_m^{3/2}}\frac{f_Mf_D^2}{f_m^2}
	\left(1+\frac{f_M^2}{f_m^2}\right)\\
	&\le&2A^2(\theta)\left[-\nu+\frac{k\nu}{2}+\frac{k\nu}{2}\right]\\
	&=&-2(1-k)\nu A^2(\theta)
\end{eqnarray*}	using \eqref{B.1} and $4A(\theta)\frac{\lambda_M}{\lambda_m^{3/2}}\frac{f_Mf_D^2}{f_m^2}
\left(1+\frac{f_M^2}{f_m^2}\right)\le \frac{k\nu}{2}$ because of \eqref{delta}.		The rest of the proof is as in \cite{giesl04}.
			
\end{appendix}

\section*{References}

\bibliography{mybibfile}
\end{document}